\documentclass{article}
\usepackage{amsthm}
\newtheorem{theorem}{Theorem}
\newtheorem{proposition}{Proposition}
\newtheorem{lemma}{Lemma}
\usepackage{fullpage}

\usepackage{amssymb,amsmath}

\usepackage{hyperref}

\usepackage{graphics}
\usepackage{epstopdf}

\usepackage{algorithm}
\usepackage{algorithmic}
\usepackage{epstopdf}
\usepackage{graphicx}
\usepackage{multirow}
\usepackage{color}
\usepackage[colorinlistoftodos,bordercolor=orange,backgroundcolor=orange!20,linecolor=orange,textsize=scriptsize]{todonotes}

\usepackage{mathtools}
\DeclarePairedDelimiter\ceil{\lceil}{\rceil}

\newcommand{\norm}[1]{\left \lVert #1 \right\rVert}
\def\<#1,#2>{\langle #1,#2\rangle}

\newcommand{\ve}[2]{\langle #1 ,  #2 \rangle}   % inner product
\newcommand{\eqdef}{\stackrel{\text{def}}{=}}
\newcommand{\R}{\mathbb{R}}
\newcommand{\Reg}{\psi}
\newcommand{\Exp}{\mathbf{E}}
\newcommand{\Prob}{\mathbf{P}}

\newcommand{\bN}{\mathbb{N}}
\newcommand{\dom}{\operatorname{dom}}

\newtheorem{coro}{Corollary}

\newtheorem{ass}{Assumption}
\newtheorem{remark}{Remark}

\author{Olivier Fercoq\thanks{LTCI,
	    CNRS, T\'el\'ecom ParisTech, Universit\'e Paris-Saclay, Paris,  France,
              (e-mail: {olivier.fercoq@telecom-paristech.fr})} \and Zheng Qu\thanks{Department of Mathematics, The University of Hong Kong,
Hong Kong, China,
              (e-mail: {zhengqu@hku.hk})}
}

\newcommand{\TheTitle}{Restarting accelerated gradient methods with a rough strong convexity estimate}

\title{{\TheTitle}}

\begin{document}

\maketitle

\begin{abstract}
We propose new restarting strategies for accelerated gradient and accelerated coordinate descent methods.
Our main contribution is to show that the restarted method has a geometric rate 
of convergence for any restarting frequency, and so it allows us to take profit of restarting even when we do not know the strong convexity coefficient.
The scheme can be combined with adaptive restarting, leading to the first 
provable convergence for adaptive restarting schemes with accelerated gradient methods.
Finally, we illustrate the properties of the algorithm on a regularized logistic regression problem and on a Lasso problem.
\end{abstract}

%\begin{keywords}
%Accelerated gradient descent; restarting strategies; unknown strong convexity; coordinate descent.
%\end{keywords}

\section{Introduction}

%TODO: present the new work

\subsection{Motivation}

The proximal gradient method aims at minimizing composite
convex functions of the form $$F(x) = f(x) + \psi(x),\enspace x\in \R^n$$
where $f$ is differentiable with Lipschitz gradient 
and $\psi$ may be nonsmooth but has an easily computable proximal operator.  For a mild additional computational cost, accelerated gradient methods transform 
the proximal gradient method, for which the optimality gap 
$F(x_k) - F(x_*)$ decreases as $O(1/k)$, into an algorithm with 
``optimal'' $O(1/k^2)$ complexity~\cite{nesterov1983method}.  
Accelerated variants include the
dual accelerated proximal gradient~\cite{Nesterov05:smooth, nesterov2013gradient}, the accelerated proximal gradient method (APG)~\cite{tseng2008accelerated}
and FISTA~\cite{beck2009fista}. 
Gradient-type methods, 
also called first-order methods, 
are often used to solve large-scale problems  because of their good scalability and easiness
of implementation that facilitates parallel and distributed computations.

In the case when 
the nonsmooth function $\psi$ is separable, which means 
that it writes as $$\psi(x) = \sum_i \psi^i(x^i),\enspace x=(x^1,\dots,x^n)\in \R^n,$$
coordinate descent methods are often considered thanks to the separability of the proximal operator of~$\psi$.
These are optimization algorithms that update only one coordinate of the
 vector of variables at each iteration, hence using partial 
 derivatives rather than the whole gradient. 
%They are quite famous for problems like truss topology design~\cite{RT:TTD2011},
%$L_1$-regularized least squares regression~\cite{WuLange:2008}, logistic regression~\cite{Lin:2010:COMS}, the dual support vector
%machines problem~\cite{Stoch-dual-Coord-Ascent}\ldots
In~\cite{Nesterov:2010RCDM}, Nesterov introduced the randomized coordinate descent method with
an improved guarantee on the iteration complexity.
He also gave an accelerated coordinate descent method
for smooth functions.
%that enjoys a rate of convergence of the order $O(1/k^2)$ on convex functions
%instead of the usual rate $O(1/k)$.
%As Nesterov recognized himself, the algorithm was not efficient but 
Lee and Sidford~\cite{lee2013efficient} introduced an efficient implementation of the method 
and Fercoq and Richt\'arik~\cite{FR:2013approx} developed the accelerated proximal coordinate descent method (APPROX) for the minimization of composite functions.

%However, like in~\cite{beck2009fista}, they decided to concentrate on the case of convex functions that are not strongly convex.

When solving a strongly convex problem, classical
(non-accelerated) gradient and coordinate descent methods
automatically have a linear rate of convergence,
i.e. $F(x_k) - F(x_*) \in O((1-\mu)^k)$
for a problem dependent $0<\mu<1$, whereas one needs to know explicitly the strong convexity parameter in order
to set accelerated gradient and accelerated coordinate descent methods to have a linear rate of convergence, 
see for instance~\cite{lee2013efficient,UniversalCatalyst,lin2014accelerated,Nesterov:2010RCDM,nesterov2013gradient}. 
Setting the algorithm with an incorrect parameter may result in a slower algorithm, sometimes even slower than if
we had not tried to set an acceleration scheme~\cite{o2012adaptive}.
This is a major drawback of the method because in general, the strong convexity parameter is
difficult to estimate. 

In the context of accelerated gradient method with unknown strong convexity parameter, Nesterov~\cite{nesterov2013gradient} proposed a restarting scheme which adaptively approximate the strong convexity parameter. The similar idea was exploited by Lin and Xiao~\cite{Lin2015} for sparse optimization.
Nesterov~\cite{nesterov2013gradient} also showed that,
instead of deriving a new method designed to work better for strongly convex functions,
one can restart the accelerated gradient method and get a linear convergence rate. However, the restarting frequency he proposed still depends explicitly on the strong convexity
of the function and so  O'Donoghue and Candes~\cite{o2012adaptive} introduced some heuristics to adaptively restart the 
algorithm and obtain good results in practice. 

%We show in this section that the same approach works for the accelerated coordinate descent method.

\subsection{Contributions}
 
In this paper, we show that we can restart accelerated  gradient and coordinate descent methods, 
including APG, FISTA and APPROX,
at {\em any} frequency and get a linearly convergent algorithm. 
The rate depends on an estimate of the strong convexity 
and we show that for a wide range of this parameter,
one obtains a faster rate than without acceleration.
In particular, we do not require the estimate
of the strong convexity coefficient to be smaller than
the actual value.
In this way, our result supports and explains the practical success of arbitrary periodic restart 
for accelerated gradient methods.

In order to obtain the improved theoretical rate, we need to define
a novel point where the restart takes place, which is a convex
combination of previous iterates. Our approach is radically different from the previous restarting schemes~\cite{Lin2015,nesterov2013gradient,o2012adaptive}, for which the evaluation of the gradient or the objective value is needed in order to verify the restarting condition. In particular, our approach can be extended to a restarted APPROX, which admits the same theoretical complexity bound as the accelerated coordinate descent methods for strongly convex functions~\cite{lin2014accelerated} and exhibits better performance in numerical experiments.

In Sections~\ref{sec:problem} and~\ref{sec:convgrad} we recall
the main convergence results for accelerated gradient methods.
In Section~\ref{sec:restart}, we present our restarting rules:
one for accelerated gradient and one for accelerated
coordinate descent. Finally, we present numerical experiments
on the lasso and logistic regression problem in Section~\ref{sec:expe}.

%\begin{itemize}
 %\item We define an accelerated coordinate descent for coordinate descent on composite problems.
%Like in~\cite{Nesterov:2010RCDM,lee2013efficient,lin2014accelerated}, the algorithm depends explicitly on
%an  estimate of the strong convexity coefficient of the function. 
%However, unlike previous works, we give a theoretical analysis in the case when this estimate is 
%larger than the actual value. This shows that this parameter is important but that convergence will still occur even in an ill setting.
%\item We propose a restarting strategy, similar to what was proposed for accelerated gradient methods~\cite{o2012adaptive},
%that leads to a comparable iteration complexity.
%\item We develop a heuristic for an adaptive restarting of the algorithm, inspired by~\cite{o2012adaptive},
%which allows us to take profit of restarting even when we do not know the strong convexity coefficient.
%\end{itemize}

\section{Accelerated gradient schemes}
\label{sec:problem}

\subsection{Problem and assumptions}
For simplicity we present the algorithm in  coordinatewise form. The extension to blockwise setting follows naturally (see for instance~\cite{FR:2013approx}).
We consider the following optimization problem:
\begin{equation}\label{eq-prob}
\begin{array}{ll}
 \mathrm{minimize} & F(x):=f(x)+\psi(x)\\
\mathrm{subject~to~} & x=(x^1,\dots,x^n)\in \R^n,
\end{array}
\end{equation}
where $f:\R^n\rightarrow \R$ is a differentiable convex function and $\psi:\R^n\rightarrow \R\cup\{+\infty\}$ is a closed convex  and separable function:
$$
\psi(x)=\sum_{i=1}^n\psi^i(x^i).
$$
Note that this implies that each function $\psi^i:\R\rightarrow \R$ is closed and convex. Let $x^*$ denote a solution of~\eqref{eq-prob}. We further assume that for each positive vector $v=(v_1,\dots,v_n)$, there is a constant $\mu_F(v)>0$ such that
\begin{align}\label{a:strconv}
F(x)\geq F(x^*) +\frac{\mu_F(v)}{2}\|x-x^*\|_v^2,
\end{align}
where $\|\cdot\|_v$ denotes the weighted  Euclidean norm in $\R^n$ defined by:
\[\|x\|_v^2 \eqdef \sum_{i=1}^n v_i (x^i)^2.\]
\begin{remark} \normalfont
Note that~\eqref{a:strconv} is weaker than the usual strong convexity assumption on $F$. In particular, ~\eqref{a:strconv} does not imply that the objective function $F$ is strongly convex. However, ~\eqref{a:strconv} entails the uniqueness of solution to~\eqref{eq-prob} and by abuse of language, we refer to~\eqref{a:strconv} as the strong convexity assumption.
\end{remark}
\subsection{Accelerated gradient schemes}

In this paper, we are going to restart accelerated gradient schemes.
We will concentrated on three versions that we will call FISTA (Fast Iterative Soft Thresholding Algorithm)~\cite{beck2009fista},
APG (Accelerated Proximal Gradient) \cite{tseng2008accelerated}
and APPROX (Accelerated Parallel and PROXimal coordinate descent) \cite{FR:2013approx}. 
In the following,  $\nabla f(y_k)$ denotes the gradient of $f$ at point $y_k$ and  $\nabla_i f(y_k)$ denotes the partial derivative of $f$ at point $y_k$ with respect to the $i$th coordinate.   $\hat S$ is  a random subset of $[n]:= \{1,2,\dots,n\}$ with the property that $\Prob(i \in \hat{S})=\Prob(j\in \hat{S})$ for all $i,j \in [n]$ and $\tau=\Exp[|\hat S|]$.
\begin{algorithm}
\begin{algorithmic}[1]
\STATE{Choose $x_0\in \dom \psi$. Set $\theta_0 = 1$ and $z_0 = x_0$}.
\FOR{$k\geq 0$} 
\STATE $y_k=(1-\theta_k)x_k+\theta_k z_k$ \\
\STATE
$x_{k+1} = \arg\min_{x\in \R^n} \big\{\< \nabla f(y_k), x-y_k>+\frac{1}{2} \|x-y_k\|^2_v+\psi(x) \big\}$ \\
\STATE $z_{k+1}=z_k+ \frac{1}{\theta_k}(x_{k+1}-y_k)$\\
\STATE $\theta_{k+1}=\frac{\sqrt{\theta_k^4 + 4 \theta_k^2} - \theta_k^2}{2}$
\ENDFOR
\end{algorithmic}
\caption{FISTA}
\end{algorithm}

\begin{algorithm}
\begin{algorithmic}[h!]
\STATE{Choose $x_0\in \dom \psi$. Set $\theta_0 = 1$ and $z_0 = x_0$}.
\FOR{$k\geq 0$} 
\STATE $y_k=(1-\theta_k)x_k+\theta_k z_k$ \\
\STATE
$
z_{k+1}
=\arg\min_{z\in \R^n} \big\{\< \nabla f(y_k), z-y_k>+\frac{\theta_k}{2} \|z-z_k\|^2_v+\psi(z) \big\}
$ \\
\STATE $x_{k+1}=y_k+ \theta_k(z_{k+1}-z_k)$\\
\STATE $\theta_{k+1}=\frac{\sqrt{\theta_k^4 + 4 \theta_k^2} - \theta_k^2}{2}$
\ENDFOR
\end{algorithmic}
\caption{APG}
\end{algorithm}

\begin{algorithm}
\begin{algorithmic}[h!]
\STATE{Choose $x_0\in \dom \psi$. Set $\theta_0 = \frac{\tau}{n}$ and $z_0 = x_0$}.
\FOR{$k\geq 0$} 
\STATE $y_k=(1-\theta_k)x_k+\theta_k z_k$ \\
\STATE Randomly generate $S_k \sim \hat{S}$\\
\FOR{ $i \in S_k$}
\STATE
$
z_{k+1}^i
=\arg\min_{z\in \R} \big\{\< \nabla_i f(y_k), z-y_k^i>+\frac{\theta_kn v_i}{2 \tau} |z-z_k^i|^2+\psi^i(z) \big\}
$ \\
\ENDFOR
\STATE $x_{k+1}=y_k+\frac{n}{\tau}\theta_k(z_{k+1}-z_k)$\\
\STATE $\theta_{k+1}=\frac{\sqrt{\theta_k^4 + 4 \theta_k^2} - \theta_k^2}{2}$
\ENDFOR
\end{algorithmic}
\caption{APPROX}
\end{algorithm}
We have written the algorithms in a unified framework to emphasize their similarities.
Practical implementations usually consider only two variables: $(x_k, y_k)$ for FISTA,
$(y_k, z_k)$ for APG and $(z_k, w_k)$ where $w_k=\theta_{k-1}^{-2}(x_k - z_k)$ for APPROX. One may also consider $t_k = \theta_k^{-1}$ instead of $\theta_k$.

The update in FISTA, APG or APPROX employs a positive vector $v\in \R^n$. To guarantee the convergence of the algorithm, the positive vector $v$ should satisfy the so-called expected separable overapproximation (ESO) assumption, developed in~\cite{FR:spcdm,RT:PCDM} for the study of parallel coordinate descent methods.
\begin{ass}[ESO] \label{ass:ESO} 
We write $(f,\hat{S})\sim \mathrm{ESO}(v)$ if
\begin{equation}\label{eq:ESO}\Exp\left[f(x+h_{[\hat{S}]})\right] \leq f(x) + \frac{\tau}{n}\left(\ve{\nabla f(x)}{h} + \frac{1}{2}\|h\|_{v}^2\right), \qquad x,h \in \R^n.
\end{equation}
where for $h=(h^1,\dots, h^n)\in \R^n$ and $S\subset [n]$, $h_{[S]}$ is defined as: \[ h_{[S]} \eqdef \sum_{i\in S} h^i e_i ,\] with $e_i$ being the $i$th standard basis vectors in $\R^n$.
\end{ass}
We require that the positive vector $v$ used in APPROX satisfy~\eqref{eq:ESO} with respect to the sampling $\hat S$ used. Note that FISTA shares the same constant $v$ with APG and APG can be seen as a special case of APPROX when $\hat S=[n]$. Therefore the positive vector $v$ used in FISTA and APG  should then satisfy:
\begin{equation*}f(x+h)
 \leq f(x) + \ve{\nabla f(x)}{h} + \frac{1}{2}\|h\|_{v}^2, \qquad x,h \in \R^n,
\end{equation*}
which is nothing but a Lipschitz condition on the gradient of $f$.  
In other words, the vector $v$ used in FISTA and APG is just the Lipschitz constant of $\nabla f$, given a diagonal scaling that may
be chosen to improve the conditioning of the problem. 

When in each step we update only one coordinate, we have $\tau=1$ and~\eqref{eq:ESO} reduces to:
\begin{equation}\label{eq:ESO-2}
\frac{1}{n}\sum_{i=1}^n  f(x+h^ie_i) \leq f(x) + \frac{1}{n}\left(\ve{\nabla f(x)}{h} + \frac{1}{2}\|h\|_{v}^2\right), \qquad x,h \in \R^n.
\end{equation}
It is easy to see that in this case the vector $v$ corresponds to the coordinate-wise Lipschitz constants of $\nabla f$, see e.g.~\cite{Nesterov:2010RCDM}.  Explicit formulas for computing admissible $v$ with respect to more general sampling $\hat S$ can be found in~\cite{RT:PCDM,FR:spcdm, QR:acdn}.

\iffalse
Denote by $\mu_{\psi} \geq 0$ the strong convexity parameter of the function $\psi$ with respect to the norm $\|\cdot\|_v$ such that
\begin{align}\label{a-strongconvepsi}
 \psi(ty+(1-t)x)\leq t\psi(y)+(1-t)\psi(x)-\frac{\mu_{\psi}t(1-t)}{2}\|x-y\|_v^2,\qquad \forall x,y\in \R^n, \; t\in[0,1].
\end{align}
Let $\mu_f\geq 0$ be the strong convexity parameter of  $f$ with respect to the norm $\|\cdot\|_v$. As $f$ is differentiable, we have:
\begin{align}\label{a-strongconve}
 f(y)\geq f(x)+\<\nabla f(x), y-x >+\frac{\mu_f }{2}\|y-x\|_v^2,\qquad  x,y\in \R^n.
\end{align}
It can be easily shown\footnote{Indeed, letting $y=x+h_{[\hat{S}]}$ in \eqref{a-strongconve} and taking expectation in $\hat{S}$, we obtain $\Exp[f(x+h_{[\hat{S}]})]\geq f(x) + \frac{\tau}{n}\left(\ve{\nabla f(x)}{h} + \frac{\mu_f}{2}\|h\|_v^2\right).$ It remains to compare this inequality with the definition of ESO, i.e., with \eqref{eq:ESO}.} that $\mu_f\leq 1$. 
Then we denote $\mu_F(v) = \mu_f + \mu_\psi$.
\fi

\section{Convergence results for accelerated gradients methods}
\label{sec:convgrad}
In this section we review two basic convergence results of FISTA and APPROX, which will be used later to build restarted methods. We first recall the following properties on the sequence $\{\theta_k\}$. 
%\todo[inline]{To Olivier: could you give a reference for~\eqref{athetabd}? I did not find it in~\cite{tseng2008accelerated}}
\begin{lemma}
The sequence $(\theta_k)$ defined by $\theta_0 \leq 1$ and $\theta_{k+1} = \frac{\sqrt{\theta_k^4 + 4 \theta_k^2} - \theta_k^2}{2}$ satisfies
\begin{align}&\label{athetabd}\frac{1}{k+1/\theta_0} \leq \theta_k \leq \frac{2}{k+2/\theta_0}
\\\label{arectheta}
&\frac{1-\theta_{k+1}}{\theta_{k+1}^2}=\frac{1}{\theta_k^2},\enspace \forall k=0,1,\dots
\\& \theta_{k+1}\leq \theta_k, \enspace \forall k =0,1,\dots, \label{atheradecr}
\end{align}
\end{lemma}
\begin{proof}
We give the proof for completeness.
\eqref{arectheta} holds because $\theta_{k+1}$ is the unique positive square root
to the polynomial $P(X) = X^2 + \theta_k^2 X - \theta_k^2$. \eqref{atheradecr} is a direct
consequence of \eqref{arectheta}.

Let us prove \eqref{athetabd} by induction. 
It is clear that $\theta_0 \leq \frac{2}{0+2/\theta_0}$. 
Assume that $\theta_k \leq \frac{2}{k+2/\theta_0}$.
We know that $P(\theta_{k+1}) = 0$ and that $P$ is an increasing function on $[0 , +\infty]$. So we just need to show that $P\big(\frac{2}{k+1+2/\theta_0}\big)\geq 0$.
\begin{align*}
P\Big(\frac{2}{k+1+2/\theta_0}\Big) = \frac{4}{(k+1+2/\theta_0)^2} 
+ \frac{2}{k+1+2/\theta_0}\theta_k^2 - \theta_k^2
\end{align*}
As $\theta_k \leq \frac{2}{k+2/\theta_0}$ and $\frac{2}{k+1+2/\theta_0}-1 \leq 0$,
\begin{equation*}
P\Big(\frac{2}{k+1+2/\theta_0}\Big) \geq \frac{4}{(k+1+2/\theta_0)^2} 
+ \Big(\frac{2}{k+1+2/\theta_0} - 1\Big) \frac{4}{(k+2/\theta_0)^2}
 = \frac{1}{(k+1+2/\theta_0)^2 (k+2/\theta_0)^2} \geq 0.
\end{equation*}

For the other inequality,  $\frac{1}{0+1/\theta_0} \leq \theta_0$.
%Let us assume that $\frac{1}{k+1/\theta_0} \leq \theta_k \leq \frac{2}{k+2/\theta_0}$.
We now assume that $\theta_k \geq \frac{1}{k+1/\theta_0}$ but
that $\theta_{k+1} < \frac{1}{k+1+1/\theta_0}$. Then, using \eqref{arectheta} and the inequality we just proved we have
\begin{align*} 
(k+1+1/\theta_0)^2 < \frac{1}{\theta_{k+1}^2}  \overset{\eqref{arectheta}}{=} \frac{1}{\theta_k^2} + \frac{1}{\theta_{k+1}} \leq (k+1/\theta_0)^2 + (k+1+2/\theta_0).
\end{align*}
This is equivalent to
\begin{equation*}
2 (k+1/\theta_0) + 1 < k + 1 + 2/\theta_0
\end{equation*}
which obviously does not hold for any $k \geq 0$.
So $\theta_{k+1} \geq \frac{1}{k+1+1/\theta_0}$.
\end{proof}

\begin{proposition}
\label{prop:fista_basic}
The iterates of FISTA satisfy for all $k \geq 1$,
\begin{equation}
\label{eq:itcompl_fista}
\frac{1}{\theta_{k-1}^2}(F(x_{k}) - F(x_*)) + \frac{1}{2}\norm{z_{k} - x_*}_v^2 \leq  \frac{1}{2}\norm{x_{0} - x_*}_v^2
\end{equation}
and 
\begin{equation}
\label{eq:stability_fista}
\frac{1}{2}\norm{x_k - x_*}_v^2  \leq  \frac{1}{2}\norm{x_{0} - x_*}_v^2 
\end{equation}

\end{proposition}
\begin{proof}
Since $z_{k+1} = z_{k} + \theta_{k}^{-1}(x_{k+1} - y_{k}) = \theta_{k}^{-1} x_{k+1} - (\theta_{k}^{-1} - 1) x_{k}$, Inequality \eqref{eq:itcompl_fista} is a simple consequence of Lemma~4.1 in~\cite{beck2009fista} (Note that we have a shift of indices for our variables $(x_{k+1}, z_{k+1})$ vs $(x_k, u_k+x_*)$ in~\cite{beck2009fista}). For the second inequality, 
we first remark that as the left term in \eqref{eq:itcompl_fista} is the sum of two
nonnegative summands, each of them is smaller than the right hand side. Hence, for all $k$, 
\[
\frac 12 \norm{z_{k} - x_*}_v^2 \leq \frac 12 \norm{x_{0} - x_k}_v^2 .
\]
Then, we rewrite $z_{i+1} = z_{i} + \theta_{i}^{-1}(x_{i+1} - y_{i})$ as $x_{i+1} = \theta_i z_{i+1} + (1- \theta_i) x_i$. Recursively applying this convex equality, we get 
that there exists $\gamma_k^i \geq 0$ such that $\sum_{i=0}^k \gamma_k^i = 1$ and $x_k =  \sum_{i=0}^k \gamma_k^i z_i$.
\begin{align*}
\frac 12 \norm{x_k - x_*}_v^2 = \frac{1}{2} \norm{\sum_{i=0}^n \gamma_k^i (z_i - x_*)}_v^2  \leq \frac{1}{2}\sum_{i=0}^n \gamma^i_k \norm{z_i - x_*}_v^2 \leq \frac{1}{2}\sum_{i=0}^n \gamma^i_k  \norm{x_{0} - x_k}_v^2
\end{align*}
and we conclude using $\sum_{i=0}^k \gamma^i_k= 1$.
\end{proof}

\begin{proposition}
\label{prop:approx_basic}
The iterates of APPROX satisfy for all $k \geq 1$,
\begin{equation}
\label{eq:itcompl_approx}
\frac{1}{\theta_{k-1}^2}\Exp[F(x_{k}) - F(x_*)] + \frac{1}{2\theta_0^2}\Exp[\norm{z_{k} - x_*}_v^2] \leq \frac{1 - \theta_0}{\theta_{0}^2}(F(x_{0}) - F(x_*)) + \frac{1}{2\theta_0^2}\norm{x_{0} - x_*}_v^2 
\end{equation}
and 
\begin{align}
\label{eq:stability_approx}
\frac{1 - \theta_0}{\theta_{0}^2}&\Exp[F(x_{k}) - F(x_*)]  + \frac{1}{2\theta_0^2} \Exp[\norm{x_k - x_*}_v^2] \leq  \frac{1 - \theta_0}{\theta_{0}^2}(F(x_{0}) - F(x_*))  \notag \\
 & + \frac{1}{2\theta_0^2}\norm{x_{0} - x_*}_v^2 - \sum_{i=0}^{k-1}  \frac{ \gamma^i_k}{\theta_{i-1}^2}\Exp[F(x_{i}) - F(x_*)] -  \left(\frac{1}{\theta_0 \theta_{k-1}}  -\frac{1 - \theta_0}{\theta_{0}^2} \right)\Exp[F(x_{k}) - F(x_*)]
\end{align}
where $\frac{1}{\theta_{-1}^2} := \frac{1-\theta_0}{\theta_0^2}$ and $\gamma^i_k$ is defined recursively by setting $\gamma_{0}^0=1$, $\gamma_{1}^0 = 0$, $\gamma_{1}^1 = 1$ and for $k\geq 1$,
\begin{equation}\label{eq:gammas}\gamma_{k+1}^i = \begin{cases}
(1-\theta_{k})\gamma_{k}^i, & i = 0,\dots,k-1,\\
\theta_{k} (1-\frac{n}{\tau}\theta_{k-1})+\frac{n}{\tau}(\theta_{k-1}-\theta_{k}),  & i=k,\\
\tfrac{n}{\tau}\theta_{k}, & i=k+1. 
\end{cases}\end{equation}
\end{proposition}
\begin{proof}
Inequality \eqref{eq:itcompl_approx} is just Theorem~3 of~\cite{FR:2013approx}. For the second inequality, 
we first isolate $\Exp[\norm{z_{k} - x_*}_v^2]$ in \eqref{eq:itcompl_approx}. For all $k\geq 0$, 
\begin{equation}\label{a:isolatezk}
\frac{1}{2\theta_0^2} \Exp[\norm{z_{k} - x_*}_v^2] \leq \frac{1 - \theta_0}{\theta_{0}^2}(F(x_{0}) - F(x_*)) + \frac{1}{2\theta_0^2}\norm{x_{0} - x_*}_v^2 - \frac{1}{\theta_{k-1}^2}\Exp[F(x_{k}) - F(x_*)] .
\end{equation}
Then, we use the fact, proved in Lemma 2 of \cite{FR:2013approx}, that  $\gamma_k^i \geq 0$, $\sum_{i=0}^k \gamma_k^i = 1$  and $x_k =  \sum_{i=0}^k \gamma_k^i z_i$. Therefore,
\begin{align}\label{a:xtoz}
\frac{1}{2\theta_0^2}\Exp[\norm{x_k - x_*}_v^2] =\frac{1}{2\theta_0^2}\Exp\left[\norm{\sum_{i=0}^k \gamma_k^i (z_i - x_*)}_v^2\right]  \leq \frac{1}{2\theta_0^2}\sum_{i=0}^k \gamma^i_k \Exp[ \norm{z_i - x_*}_v^2] 
\end{align}
Plugging~\eqref{a:isolatezk} into~\eqref{a:xtoz} we get:
\begin{align*}& 
\frac{1}{2\theta_0^2}\Exp[\norm{x_k - x_*}_v^2] \\ 
&\leq  \sum_{i=0}^{k} \gamma^i_k \left( \frac{1 - \theta_0}{\theta_{0}^2}(F(x_{0}) - F(x_*)) + \frac{1}{2\theta_0^2}\norm{x_{0} - x_*}_v^2 
  - \frac{1}{\theta_{i-1}^2}\Exp[F(x_{i}) - F(x_*)]  \right)  \label{eq:oldzs} 
\\  &=  \frac{1 - \theta_0}{\theta_{0}^2}(F(x_{0}) - F(x_*)) + \frac{1}{2\theta_0^2}\norm{x_{0} - x_*}_v^2 
  -\sum_{i=0}^{k-1} \frac{\gamma_k^i}{\theta_{i-1}^2}\Exp[F(x_{i}) - F(x_*)] - \frac{\gamma_k^k}{\theta_{k-1}^2}\Exp[F(x_{k}) - F(x_*)]
\end{align*}
and we deduce~\eqref{eq:stability_approx} using $\gamma_{k}^k=\theta_{k-1}/\theta_0$.
\end{proof}

\begin{remark} \normalfont
The strength of these propositions is that they are independent of the strong convexity parameter. Indeed,  APG, FISTA and APPROX work for non-strongly convex minimization.
\end{remark}

\begin{remark} \normalfont
As APPROX generalizes APG, we have covered all three algorithms in the two propositions. A remarkable feature is that
the result for FISTA and APG are exactly the same even though 
the algorithms are different. 
%As a consequence, from now on, we will write the proofs only for APPROX.
\end{remark}

\section{Restarted gradient methods}
\label{sec:restart}
The basic tool upon which we build our restarting rule is
a contraction property.
We first present two restarting rules that require a special condition in order to guarantee the linear convergence. 
Then we present new rules that are more complex
but are always certified to give a linearly convergent algorithm.
\subsection{Conditional restarting}

The first rule is an extension of the ``optimal fixed restart''
of \cite{nesterov2013gradient,o2012adaptive} to FISTA and APPROX.

\begin{proposition}[Conditional restarting at $x_k$]
Let $(x_k, z_k)$ be the iterates of FISTA or APPROX applied to~\eqref{eq-prob}. We have
\[
\Exp[F(x_k) - F(x_*)]\leq
\theta_{k-1}^2 \left(\frac{1 - \theta_0}{\theta_{0}^2} + \frac{\theta_0^2}{\mu_F(v)} \right)(F(x_{0}) - F(x_*)) .
\]
Moreover, given $\alpha < 1$, if 
\begin{align}\label{a:conditionalresxk}
k \geq \frac{2}{\theta_0}\left(\sqrt{\frac{1+\mu_F(v)}{\alpha\mu_F(v)}}-1\right) + 1 ,
\end{align}
then $\Exp[F(x_k) - F(x_*)] \leq \alpha (F(x_{0}) - F(x_*))$.
\label{prop:restart_x}
\end{proposition}
\begin{proof}
By~\eqref{eq:itcompl_fista} and~\eqref{eq:itcompl_approx}, the following holds for the iterates of FISTA ($\theta_0=1$) and APPROX:
\begin{align*}
\Exp[F(x_k) - F(x_*)]& {\leq}
\theta_{k-1}^2 \left(\frac{1 - \theta_0}{\theta_{0}^2}(F(x_{0}) - F(x_*)) + \frac{1}{2\theta_0^2}\norm{x_{0} - x_*}_v^2 \right) \\
& \overset{\eqref{a:strconv}}{\leq} 
\theta_{k-1}^2 \left(\frac{1 - \theta_0}{\theta_{0}^2} + \frac{1}{\mu_F(v)\theta_0^2} \right)(F(x_{0}) - F(x_*)). 
\end{align*}
Condition~\eqref{a:conditionalresxk} is equivalent to:
$$
\frac{4}{(k-1+2/\theta_0)^2}\left(\frac{1 }{\theta_{0}^2} + \frac{1}{\mu_F(v)\theta_0^2} \right) \leq \alpha,
$$
and we have the contraction using~\eqref{athetabd}.
\end{proof}
\begin{remark} \normalfont
Notice that the restarting rule~\eqref{a:conditionalresxk} requires to know a lower bound on the strong convexity coefficient of $F$.
\end{remark}
The next restarting rule is built upon a comparison condition and does not rely on any estimation of $\mu_F(v)$.
\begin{proposition}[Conditional restarting at $z_k$]
Let $(x_k, z_k)$ be the iterates of FISTA or APG applied to~\eqref{eq-prob}.
If $F(z_k) \leq F(x_k)$, then 
\begin{equation}
 \frac{1}{2}\norm{z_{k} - x_*}_v^2 \leq \frac{1}{\left(1+\frac{\mu_F(v)}{\theta_{k-1}^2} \right)}\frac{1}{2}\norm{z_{0} - x_*}_v^2.
\end{equation}
\label{prop:restart_z}
\end{proposition}
\begin{proof}
By~\eqref{eq:itcompl_fista} and~\eqref{eq:itcompl_approx}, the following holds for the iterates of FISTA and APG ($\theta_0=1$):
\begin{equation}\label{eq:FISTAAPG}
\frac{1}{\theta_{k-1}^2} (F(z_{k}) - F(x_*)) + \frac{1}{2}\norm{z_{k} - x_*}_v^2 \leq \frac{1}{2}\norm{x_{0} - x_*}_v^2 
\end{equation}
By~\eqref{a:strconv}, we get
\[
\left(\frac{\mu_F(v)}{2 \theta_{k-1}^2} + \frac{1}{2}\right)\norm{z_{k} - x_*}_v^2 \leq \frac{1}{2}\norm{x_{0} - x_*}_v^2= \frac{1}{2}\norm{z_{0} - x_*}_v^2. %\qedhere
\]
\end{proof}
Here, we do not need to know $\mu_F(v)$ but we need to wait for $F(z_k)$ to be smaller than $F(x_k)$. This event does happen
 sometimes but there is no guarantee for it to happen when minimizing a given function. Hence we may wait for ever and never restart.

\subsection{Unconditional restarting}
In this section, we prove the main results of this paper: setting a convex combination of the past iterates as the restarting point leads to a linearly convergent restarted method. 

We first show that for full-gradient accelerated methods, an arbitrary strict convex combination of the last iterates $x_k$ and $z_k$ works.
\begin{theorem}[Restarting for FISTA and APG]
Let $(x_k, z_k)$ be the iterates of FISTA or APG applied to~\eqref{eq-prob}. Let $\sigma \in [0,1]$  and $\bar x_k = (1-\sigma) x_k + \sigma z_k$. We have
\[
\frac 12 \norm{\bar x_k- x_*}^2_v \leq \frac 12 \max\left(\sigma,1- \frac{\sigma \mu_F(v)}{\theta_{k-1}^2}\right)\norm{x_0 - x_*}^2_v.
\]
\label{th:restart_sigma}
\end{theorem}

\begin{proof}
By the definition of $\bar x_k$:
\begin{align*}
&\frac 12 \norm{\bar x_k - x_*}_v^2 {\leq} \frac{1-\sigma}{2} \norm{x_k - x_*}_v^2 +  \frac{\sigma}{2} \norm{z_k - x_*}_v^2\\
&=  \left(1-\sigma - \frac{\sigma \mu_F(v)}{\theta_{k-1}^2}\right)\frac{1}{2} \norm{x_k - x_*}_v^2 +  \frac{\sigma}{\theta_{k-1}^2}\left(\frac{\mu_F(v)}{2} \norm{x_k - x_*}_v^2+\frac{\theta_{k-1}^2}{2} \norm{z_k - x_*}_v^2\right) \\
& \overset{\eqref{a:strconv}}{\leq}  \max\left(0,1-\sigma -\frac{\sigma  \mu_F(v) }{\theta_{k-1}^2}\right)\frac{1}{2} \norm{x_k - x_*}_v^2 + \frac{\sigma}{\theta_{k-1}^2}\left(  F(x_k)-F(x_*)+\frac{\theta_{k-1}^2}{2} \norm{z_k - x_*}_v^2\right) \end{align*}
Next we apply~\eqref{eq:itcompl_fista} and~\eqref{eq:stability_fista} (the same holds for APG by taking $\theta_0=1$ in~\eqref{eq:itcompl_approx} and~\eqref{eq:stability_approx}):
\begin{align*}\frac 12 \norm{\bar x_k - x_*}_v^2 
&{\leq} 
\max\left(0,1-\sigma- \frac{\sigma\mu_F(v)}{\theta_{k-1}^2}\right)\frac{1}{2} \norm{x_0 - x_*}_v^2 +  \frac{\sigma}{2} \norm{x_0 - x_*}_v^2\\
&= \max\left(\sigma,1- \frac{\sigma\mu_F(v)}{\theta_{k-1}^2}\right)\frac{1}{2}\norm{x_0 - x_*}^2_v. %\qedhere
\end{align*}
\end{proof}

For APPROX, we need a more complex restarting point.
\begin{theorem}[Restarting for APPROX]
Let $\gamma^i_k$ be the coefficients defined in \eqref{eq:gammas} and 
\begin{equation}
\label{eq:center_approx}
\mathring x_k = \frac{1}{\sum_{i=0}^{k-1}\frac{\gamma_k^i}{\theta_{i-1}^2}+ \frac{1}{\theta_0 \theta_{k-1}} -\frac{1 - \theta_0}{\theta_{0}^2} }\left( \sum_{i=0}^{k-1}\frac{\gamma_k^i}{\theta_{i-1}^2}x_i+ \left(\frac{1}{\theta_0 \theta_{k-1}} -\frac{1 - \theta_0}{\theta_{0}^2} \right)x_k \right)
\end{equation}
be a convex combination of the $k$ first iterates of APPROX. Let $\sigma \in [0,1]$ and $\bar x_k = \sigma x_k + (1-\sigma) \mathring x_k$.
Denote $$\Delta(x) := \frac{1 - \theta_0}{\theta_{0}^2}(F(x) - F(x_*)) + \frac{1}{2\theta_0^2}\norm{x - x_*}_v^2$$
and 
\begin{align}\label{amkmu}
m_k(\mu) := \frac{\mu \theta_0^2}{1+ \mu (1 - \theta_0)}\left(\sum_{i=0}^{k-1}\frac{\gamma_k^i}{\theta_{i-1}^2}+ \frac{1}{\theta_0 \theta_{k-1}} -\frac{1 - \theta_0}{\theta_{0}^2}\right).
\end{align}
We have
\begin{align*}
\Exp[\Delta(\bar x_k)] \leq \max\left(\sigma, 1 -\sigma m_k(\mu_F(v))\right) \Delta(x_0).
\end{align*}
\label{th:restart_approx}
\end{theorem}
\begin{proof}
Note that by~\eqref{atheradecr},
$$
\frac{1}{\theta_0 \theta_{k-1}}\geq \frac{1}{\theta_0^2}\geq \frac{1 - \theta_0}{\theta_{0}^2}.
$$
Hence $\mathring x_k$ is a convex combination of $\{x_0,\dots,x_k\}$.
By~\eqref{eq:stability_approx} and the definition of $\mathring x_k$,
\begin{align*}
\Delta(x_0) & \geq \Exp[\Delta(x_k)] + \sum_{i=0}^{k-1} \frac{ \gamma^i_k}{\theta_{i-1}^2}\Exp[F(x_{i}) - F(x_*)] + \left( \frac{1}{\theta_0 \theta_{k-1}} -\frac{1 - \theta_0}{\theta_{0}^2}\right)\Exp[F(x_{k}) - F(x_*)] \\
%%%%%%%%%%%%%
& \geq  \Exp[\Delta(x_k)]  + \left(\sum_{i=0}^{k-1}\frac{\gamma_k^i}{\theta_{i-1}^2}+ \frac{1}{\theta_0 \theta_{k-1}} -\frac{1 - \theta_0}{\theta_{0}^2}\right) \Exp[ F(\mathring x_k) - F(x_*)] \\
%%%%%%%%%%%%
\end{align*}
In view of the strong convexity assumption~\eqref{a:strconv},
\begin{align*}
&\Delta(x) = \frac{1 - \theta_0}{\theta_{0}^2}(F(x) - F(x_*)) + \frac{1}{2\theta_0^2}\norm{x - x_*}_v^2 \leq \left(\frac{1 - \theta_0}{\theta_{0}^2}+ \frac{1}{\mu_F(v) \theta_0^2}\right) (F(x) - F(x_*))  
\end{align*}
Therefore,
\begin{align*}
\Delta(x_0) 
& \geq \Exp[\Delta(x_k)]  + \frac{\mu_F(v)\theta_0^2 }{1 + \mu_F(v) (1-\theta_0)}  \left(\sum_{i=1}^{k-1}\frac{\gamma_k^i}{\theta_{i-1}^2}+\frac{1}{\theta_0 \theta_{k-1}} -\frac{1 - \theta_0}{\theta_{0}^2}\right) \Exp[ \Delta(\mathring x_k) ] \\
&\overset{\eqref{amkmu}}{ =} \Exp[\Delta(x_k)]  + m_k(\mu_F(v)) \Exp[ \Delta(\mathring x_k) ] 
\end{align*}
Moreover, using \eqref{eq:stability_approx} again, 
we can easily see that $\Exp[\Delta(x_i)] \leq \Delta(x_0)$ 
for all $i$ and thus  $\Exp[ \Delta(\mathring x_k) ] \leq \Delta(x_0)$.
Let us now consider $\bar x_k = \sigma x_k + (1-\sigma) \mathring x_k$.
\begin{align*}
\Exp[\Delta(&\bar x_k)] \leq \sigma \Exp[\Delta(x_k)] + (1-\sigma) \Exp[\Delta(\mathring x_k)] \\
&= \sigma \Exp[\Delta(x_k)] + \sigma m_k(\mu_F(v)) \Exp[\Delta(\mathring x_k)]
+  \left(1-\sigma - \sigma m_k(\mu_F(v)) \right)\Exp[\Delta(\mathring x_k)] \\
&\leq \sigma \left(\Exp[\Delta(x_k)] +  m_k(\mu_F(v)) \Exp[\Delta(\mathring x_k)]\right)
+ \max\left( 0,  1- \sigma - \sigma m_k(\mu_F(v)) \right)\Exp[\Delta(\mathring x_k)] \\
& \leq \sigma  \Delta(x_0)
+ \max\left( 0,  1 - \sigma - \sigma m_k(\mu_F(v)) \right)\Delta(x_0)\\
& =   \max\left( \sigma ,1 - \sigma m_k(\mu_F(v)) \right)\Delta(x_0) %\qedhere
\end{align*}
\end{proof}

\subsection{Restarted APPROX}
We describe in Algorithm~\ref{algo:restartedFGMv2} the restarted APPROX method and give the convergence result in Theorem~\ref{th:sigmaK}.

\begin{algorithm}
\begin{algorithmic}[h!]
\STATE Choose $x_0 \in \R^n$, set $z_0=x_0$ and $\theta_0=\frac{\tau}{n}$.
\STATE Choose $\sigma \in (0,1)$ and $K \in \mathbb{N}$.
\FOR{ $k \geq 0$}
\STATE $y_k= (1-\theta_k) x_k+ \theta_k z_k$
\STATE Generate a random set of coordinates $S_k\sim \hat{S}$
\FOR{ $i \in S_k$}
\STATE $z_{k+1}^i = \arg\min_{z \in \R} \left\{ \langle \nabla_i f (y_k), z- y_k^i \rangle + \frac{\theta_k n v_i }{2\tau} \|z- z_k^i \|_{v}^2 + \Reg^i(z)\right\}$
\ENDFOR 
\STATE $x_{k+1} =y_k +\frac{n}{\tau}\theta_k(z_{k+1}- z_k )$
\STATE $\theta_{k+1}= \frac{\sqrt{\theta_k^4+4\theta_k^2}-\theta_k^2}{2}$
\IF{ $k \equiv 0 \; \mathrm{mod} \; 
%\left \lfloor %\frac{2ne}{\tau}(1-\frac{\tau}{n}+\frac{1}{\mu_F(v)})^{1/2}-\frac{n-\tau}{\tau} %\right \rfloor 
K
$}\label{line:testrestarting2}
\STATE $x_{k+1} \leftarrow \bar x_{k+1}$ \quad ($\bar x_{k+1}$ is defined in Theorem~\ref{th:restart_approx})
\STATE $z_{k+1} \leftarrow \bar x_{k+1}$
\STATE $\theta_{k+1} \leftarrow \theta_0$
\ENDIF
\ENDFOR 
\end{algorithmic}
\caption{APPROX+restart}
\label{algo:restartedFGMv2}
\end{algorithm}
\begin{remark} \normalfont
For this restarting rule to be useful, we need to be able
to compute $\mathring{x}_k$ efficiently, in particular without computing $x_i$ for $ i < k$.
A way to do this is to use the variable $w_i = \theta_{i-1}^{-2}(x_i - z_i)$, which is maintained up-to-date in the algorithm:
\[
\sum_{i=0}^{k-1}\frac{\gamma_k^i}{\theta_{i-1}^2}x_i = 
\sum_{i=0}^{k-1}\frac{\gamma_k^i}{\theta_{i-1}^2}z_i + \gamma_k^i w_i.
\]
Then, we can compute the sum using cumulative updates like
 in~\cite{dang2015stochastic}. We develop this idea in Appendix~\ref{sec:impl}.
\end{remark}
\begin{theorem}\label{th:sigmaK}
Let us choose $K \in \mathbb{N}$ and $\sigma \in (0,1)$ as we wish. Using the notation defined in Theorem~\ref{th:restart_approx}, the iterates of Algorithm~\ref{algo:restartedFGMv2} satisfy for any $k \geq K$
\begin{align*}&
\Exp\left[ \Delta(x_k)\right] 
\leq   \left(\max\left( \sigma, 1 -\sigma m_K(\mu_F(v)) \right)^{1/K}\right)^{k-K}\Delta(x_0).
\end{align*}
\end{theorem}
\begin{proof}
Let us write the Euclidean division $k = m K + r$ with $r \in [0, K-1]$. Using \eqref{eq:stability_approx} and Theorem~\ref{th:restart_approx},
\begin{align*}
\Delta(x_k) &\leq \Delta(x_{mK})  = \Delta(\bar x_{mK}) \leq \max\left( \sigma, 1 -\sigma m_K(\mu_F(v)) \right) \Delta(x_{(m-1) K}) \\
& \leq  \max\left( \sigma, 1 -\sigma m_K(\mu_F(v)) \right)^m \Delta(x_0)  \\
& =  \left(\max\left( \sigma, 1 -\sigma m_K(\mu_F(v)) \right)^{1/K}\right)^{k-r} \Delta(x_0)\\
& \leq  \left(\max\left( \sigma, 1 -\sigma m_K(\mu_F(v)) \right)^{1/K}\right)^{k-K} \Delta(x_0) %\qedhere
\end{align*}
%If we choose $K$ and $\sigma$ such that $\theta_{K-1} \approx \frac{2}{K} = \sqrt{\sigma}$, we get a rate of convergence equal to
%\[
%\left(\frac{\sigma + \max(0, \sigma - \mu_F(v))}{ 2 \sigma}\right)^{\sqrt{\sigma}/2}
%\]
%which is best at $\sigma = \mu_F(v)$ but not that bad otherwise. It is a little worse than the gradient descent's one at $\sigma = 0$ which suggests that some improvement in the bound is still possible.
\end{proof}
Theorem~\ref{th:sigmaK} shows that Algorithm~\ref{algo:restartedFGMv2} is linearly convergent
with respect to arbitrary convex combination coefficient $\sigma\in(0,1)$ and arbitrary restarting period $K \in \mathbb{N}$. This implies that we can always get linear convergence  without any information on the strong convexity parameter of the objective function. The next proposition provides an estimation on the rate of convergence given a guess $\mu$ on the parameter $\mu_F(v)$ and a particular choice of $\sigma$ and $K$.

%In Appendix~\ref{sec:estimation}, we give bounds on $m_K(\mu_F(v))$ that helped us choose $K$ and $\sigma$ according to a guess $\mu$ on the value of the strong convexity coefficient $\mu$.
 
\begin{proposition}
Let $\mu \in (0, 1]$.  Choose \begin{align} \label{aK}
K=\ceil*{\frac{2\sqrt{3}}{\theta_0}\sqrt{1+\frac{1}{\mu}}-\frac{2}{\theta_0}+1 },
\end{align}
and 
\begin{align} \label{asigma}
\sigma=\frac{1}{1+m_K(\mu)}.
\end{align}
The iterates of Algorithm~\ref{algo:restartedFGMv2} satisfy for any $k \geq K$
\[
\Exp[\Delta(x_k)] \leq \left(1- \min \left(\frac{\mu_F(v)}{\mu} ,1\right)\frac{1+\mu\theta_0}{2+\mu} \right) ^{\frac{k\theta_0 \sqrt{\mu}}{2\sqrt{3}\sqrt{(1+\mu)}}-1}\Delta(x_0).
\]
\label{prop:choose_sigma_and_K}
\end{proposition}
\begin{proof}
This is a direct corollary of Proposition~\ref{prop:choose_sigma_and_K_appendix} by taking $\lambda=1+\mu$, presented in Appendix~\ref{sec:estimation}.
\end{proof}

 Let us have additional insight of Proposition~\ref{prop:choose_sigma_and_K}.
\begin{coro}\label{corok}
Denote $D(x_0) = \theta_0^2 \Delta(x_0) = (1 - \theta_0
)(F(x_0) - F(x_*)) + \frac{1}{2}\norm{x_0 - x_*}_v^2$. Let $\mu \in (0, 1]$.  Choose $K$ and $\sigma$ as in~\eqref{aK}
and~\eqref{asigma}. Then for 
$$
k\geq \frac{n}{\tau}\left(6\sqrt{6}\max\left(\frac{1}{\sqrt{\mu}}, \frac{\sqrt{\mu}}{\mu_F(v)}\right)
\log\left(\frac{D(x_0)}{\epsilon}\right) + 2\sqrt{3} \sqrt{1+\frac{1}{\mu}}\right),
$$
we have
$$
(1 - \theta_0)(F(x_k) - F(x_*)) + \frac{1}{2}\norm{x_k - x_*}_v^2 \leq \epsilon.
$$
\end{coro}
The proof of Corollary~\ref{corok} is deferred to Appendix~\ref{sec:estimation}. We therefore showed that for any strong convexity estimator $\mu\in (0,1]$,  the iteration complexity of Algorithm~\ref{algo:restartedFGMv2} is on the order of
$$
O\left(\frac{n}{\tau} \max\left(\frac{1}{\sqrt{\mu}}, \frac{\sqrt{\mu}}{\mu_F(v)}\right) \log(1/\epsilon)  \right)
=\left\{\begin{array}{ll}
O\left(\frac{n}{\tau\sqrt{\mu}} \log(1/\epsilon) \right) & \text{if~~} \mu \leq \mu_F(v)
\\ O\left(\frac{n\sqrt{\mu}}{\tau\mu_F(v)} \log(1/\epsilon)  \right) & \text{if~~} \mu >\mu_F(v)
\end{array}\right.
$$
where the $O$ notation hides logarithms of problem dependent constants and universal constants. Recall that for coordinate descent methods~\cite{Nesterov:2010RCDM,RT:UCDC}, the iteration complexity bound is
$$
O\Big(\frac{n}{\tau \mu_F(v)}\log(1/\epsilon)\Big).
$$
Therefore, if $\mu$ is an upper bound on $\mu_F(v)$, our iteration complexity bound improves over that of the randomized coordinate descent method~\cite{lin2014accelerated,shalev2013accelerated} by a factor of $\sqrt{\mu}$; 
if $\mu$ is a lower bound on $\mu_F(v)$ such that $\mu_F(v) \leq \sqrt{\mu} \leq \sqrt{\mu_F(v)}$, we also obtain an improved complexity bound; 
only if $\mu < \mu_F^2$ is our bound worse. This observation will be  illustrated in Section~\ref{sec:expe}

\begin{remark} \normalfont
The theorems presented in this section can easily
be combined with an adaptive restart strategy. 
We just need to define an interval $[\underline{K}, \bar{K}]$
and allow the adaptive restart only if $k \in [\underline{K}, \bar{K}]$. Then if $k = \bar K$ we force the restart.
We obtain a linear convergence rate where the rate is
given by the worst case in the interval.
\end{remark}

\section{Numerical experiments}
\label{sec:expe}

\subsection{Illustration of the theoretical bounds}

We first illustrate the theoretical rate we have found.
On Figure~\ref{fig:rate_vs_mu}, we can see that
restarted APPROX has a better rate of convergence than 
vanilla proximal coordinate descent for a wide range of
estimates of the strong convexity. Indeed, in this example
with $n=10$ and $\mu_F(v) = 10^{-5}$, one can take $1.6\;10^{-9} \leq \mu \leq 0.04$.
Note that this shows that even if the estimate $\mu$ is much larger than the 
true strong convexity coefficient $\mu_F(v)$, we already see an improved rate.
Yet, of course the closer $\mu$ is to $\mu_F(v)$, the faster the algorithm will be.

On Figure~\ref{fig:rate_vs_muF}, we fix the estimate of the strong convexity
as $\mu = 10^{-3}$ and we plot the rate of convergence of the
method for $\mu_F(v) \in [10^{-9}, 1]$.

\begin{figure}
\centering
\includegraphics[width=0.6\linewidth]{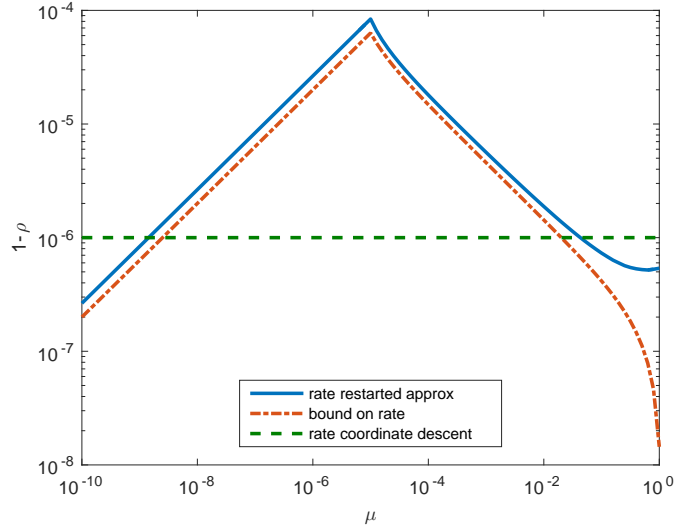}
\caption{Comparison of the rates of coordinate descent and restarted APPROX
when $\tau=1$, $n=10$ and $\mu_F(v) = 10^{-5}$. Given an estimate $\mu$ of $\mu_F(v)$,
we have chosen $\lambda =1$ and $K$ and $\sigma$ as in Proposition~\ref{prop:choose_sigma_and_K}.
The blue solid line is the rate given by Theorem~\ref{th:restart_approx} and 
the red dash-dotted line is the simpler rate given in Proposition~\ref{prop:choose_sigma_and_K}.
We are plotting 1 minus the rate $\rho$ in logarithmic 
scale for a better contrast.}
\label{fig:rate_vs_mu}
\end{figure}

\begin{figure}
\centering
\includegraphics[width=0.6\linewidth]{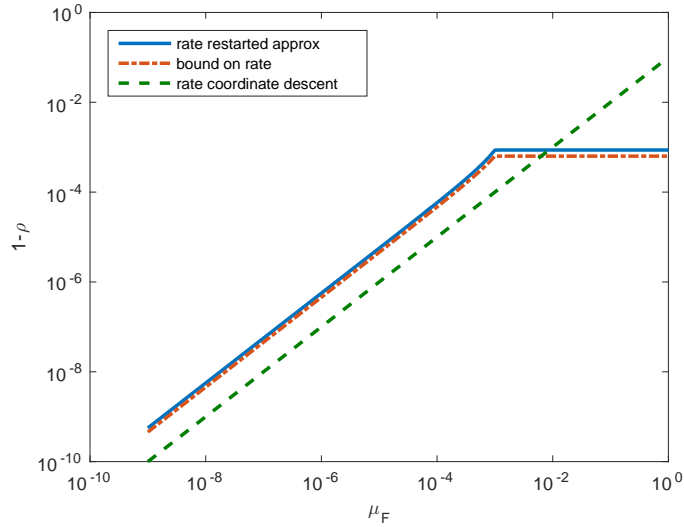}
\caption{Comparison of the rates of coordinate descent and restarted APPROX
when $\tau=1$, $n=10$ and $\mu = 10^{-3}$ (this corresponds to a restart
every $K \approx 107n$ iterations with $\sigma\approx 0.4$). For each possible value of $\mu_F(v)$,
we have computed the rate as given in Proposition~\ref{prop:choose_sigma_and_K}.
With this choice of $\mu$, restarted APPROX has a better rate than coordinate descent as soon as
$\mu_F(v) < 8.\; 10^{-3}$ and is about 5 times faster when $\mu_F(v)$ is small.}
\label{fig:rate_vs_muF}
\end{figure}

\subsection{Gradient methods}

We then present experiments on accelerated gradient methods (the case $n=\tau$).
We solve the $L^1$-regularised least squares problem (Lasso)
\[
\min_{x \in \mathbb{R}^N} \frac{1}{2} \norm{A x - b}^2_2 + \lambda \norm{x}_1
\]
on the Iris dataset where $A \in \mathbb{R}^{m \times n}$ 
is the design matrix and $b \in \mathbb{R}^m$ 
is such that $b_j = 1$ if the label is ``Iris-setosa'', $b_j = -1$ otherwise.
We chose $\lambda = \norm{A^T b}_\infty / 10$.
This dataset is rather small ($n = 4$ and $m = 8124$). 
As we can see on Table~\ref{tab:restart_comparison_accgrad}, non-accelerated proximal gradient
(underlined numbers) is faster than accelerated variants (italics).
Yet, accelerated gradient methods designed for strongly convex objectives may be faster than
both vanilla proximal gradient and basic accelerated proximal gradient. The ``true'' strong convexity coefficient is around 5.3 $10^{-4}$:
we can see that taking $\mu_\text{est}$ close to this value
leads indeed to a faster algorithm but that the algorithms
are rather stable to approximations.

Dual APG is quite efficient on this dataset but as its restarting rule
is based on a divergence detection scheme, some attention should
be paid before using intermediate solutions. Also remark
that when it is not restarted, APG exhibits the sublinear $O(1/k^2)$ rate
while FISTA seems to be take some profit of strong convexity even without restarting.

All restarting strategies seem to perform well. Note however that
APG-$\mu$ and FISTA-$\mu$ are only proved to converge linearly when
$\mu \leq \mu_F(v)$ and that the adaptive restart of~\cite{o2012adaptive}
is a heuristic restart. Indeed, with APG, the restart condition did
not happen within the first 10,000 iterations, which shows that 
this adaptive restart is not always efficient.

The rule of Theorem~\ref{th:sigmaK} is more complex and
seems to be slightly less efficient than the rule of Theorem~\ref{th:restart_sigma}. We still present this
more complex rule because it is the only one that
is proved to have a linear rate of convergence in the accelerated
coordinate descent case $n>\tau$.

\begin{table}[h]
\begin{tabular}{|l|r|r|r|r|r|r|r|r|}
\hline
$\mu_{\text{est}}$ &  1  &  0.1  & 0.01 & 0.001 & $10^{-4}$ & $10^{-5} $ & $10^{-6}$ & $10^{-8} $ \\ 
\hline
Dual APG with  &  \multirow{2}{*}{447}  &  \multirow{2}{*}{398} &   \multirow{2}{*}{265}  &  \multirow{2}{*}{156}  &  \multirow{2}{*}{162}  &  \multirow{2}{*}{163}   & \multirow{2}{*}{163}   & \multirow{2}{*}{163} \\
adaptive restart~\cite{nesterov2013gradient} &&&&&&&&\\
\hline
FISTA-$\mu$~\cite{vandenberghe2016lecture} &  \underline{751}  & 352 &  170 &  173 &  264  & 291  & 277 & 277 \\
\hline
FISTA restarted: &&&&&&&&\\   
at $x$, Proposition~\ref{prop:restart_x} &  \underline{751} &  687 &  297 & 160 & 198 &  {\em 278} &  {\em 278} & {\em 278} \\
at $z$, Proposition~\ref{prop:restart_z}  &   \underline{751} &  464 &   222 &  245 &  311  & {\em 278} &  {\em 278} & {\em 278}  \\
as Theorem \ref{th:restart_sigma}   &  633 &   274 &  168 &  211 &  {\em 278} &  {\em 278} &  {\em 278} & {\em 278} \\
as Theorem \ref{th:sigmaK}  &  801 &  477  &  181 &  232  &{\em 278} & {\em 278} &  {\em 278} & {\em 278} \\
if $F(x_{k+1}) > F(x_k)$~\cite{o2012adaptive} & \multicolumn{7}{c}{121} &\\
 \hline 
 APG-$\mu$~\cite{lin2014accelerated} & \underline{751}   &  351 & 340  &   882 & 2580  & 7453 &   $>$10000 & $>$10000 \\
 \hline
APG  restarted: &&&&&&&&\\ 
at $x$, Proposition~\ref{prop:restart_x} &  \underline{751}  & 684 & 297   &  189 & 311 &   894  & 1471 & 4488 \\
at $z$, Proposition~\ref{prop:restart_z}  &  \underline{751}  & 463 & 221   & 232  & 281   & 460   & 1415   &    4473 \\
as Theorem \ref{th:restart_sigma} &  632 &  275 & 173 & 281  & 794 &   1310  & 3977  & $>${\em 10000} \\
as Theorem \ref{th:sigmaK} & 801  &  477 &  214 &  288 & 703 & 1166  &  3494  &  $>${\em 10000} \\
if $F(x_{k+1}) > F(x_k)$~\cite{o2012adaptive}& \multicolumn{7}{c}{$>${\em 10000}} & \\
\hline
\end{tabular}
\caption{Number of iterations to reach $F(x_k) - F(x_*) \leq 10^{-10}$ with various accelerated algorithms for the 
Lasso problem on the Iris dataset.
In italics, no restart has taken place; underlined numbers means that the algorithm is equivalent to non-accelerated ISTA.}
\label{tab:restart_comparison_accgrad}
\end{table}

\subsection{Coordinate descent}

We solve the following logistic regression problem:
\begin{align}\label{alr}
\min_{x \in \mathbb{R}^N} \frac{\lambda_1}{2 \|A^\top b\|_{\infty}} \sum_{j=1}^m \log(1 + \exp(b_j a_j^\top x)) + \norm{x}_1+\frac{\lambda_2}{2}\|x\|^2\end{align}
We consider 
$$
f(x)= \frac{\lambda_1}{ 2\|A^\top b\|_{\infty}} \sum_{j=1}^m \log(1 + \exp(b_j a_j^\top x)) ,
$$
and 
$$
\psi(x)=\norm{x}_1+\frac{\lambda_2}{2}\|x\|^2.
$$
In particular, for serial sampling ($\tau=1$),~\eqref{eq:ESO} is satisfied for
\begin{align}\label{av}
v_i=\frac{\lambda_1}{ 8\|A^\top b\|_{\infty}}\sum_{i=1}^m (b_jA_{ij})^2,\enspace i=1,\dots,n.
\end{align}
Then for the latter $v$, \eqref{a:strconv} is satisfied for
\begin{align}\label{amupsi}
\mu_F(v)=
\mu_{\psi}:=\frac{\lambda_2}{\max_i v_i}.
\end{align}

Even if the logistic objective
is not strongly convex, we expect that the local curvature around the optimum is nonzero
and so, that taking $\mu > \mu_F(v) = \mu_\psi$ will be useful.
We solve~\eqref{alr} for different values of $\lambda_2$, using $v$ and $\mu_F(v)=\mu_{\psi}$ defined in~\eqref{av} and~\eqref{amupsi}.  

We compare randomized coordinate descent (CD), APCG~\cite{lin2014accelerated} and APPROX-restart (Algorithm~\ref{algo:restartedFGMv2}) with $K$ and $\sigma$ given by Proposition~\ref{prop:choose_sigma_and_K}, on the dataset rcv1. 
We run both APCG and APPROX-restart using four different values of $\mu$: $\mu_F(v)$, $10\:\mu_F(v)$, $100\:\mu_F(v)$ and 1,000$\:\mu_F(v)$. We stop the program when the duality gap is lower than $10^{-10}$ or the running time is larger than 3,000s. The results are reported in Figure~\ref{fig:d0}, where by $\mu_F$ we refer to $\mu_F(v)$ defined in~\eqref{amupsi}.

 Note that the convergence of APCG is only proved for $\mu \leq \mu_F(v)$ in~\cite{lin2014accelerated}.  In our experiments, we observed numerical issues when running APCG for several cases when taking larger $\mu$ (we were
 not able to compute the $i$th partial derivative at $y_k = \rho_k w_k + z_k$ because $\rho_k$ had reached the 
 double precision float limit).  Such cases can be identified in the plots if the line corresponding to APCG 
 stops abruptly before the time limit (3000s) with a precision worse than $10^{-10}$.
On all the experiments, restarted APPROX is faster or much faster than APCG. Moreover, it is stable for any restarting frequency while APCG may
fail if one is too optimistic when setting the strong convexity estimate.

\begin{figure}[htbp]
\centering
\includegraphics[width=0.32\linewidth]{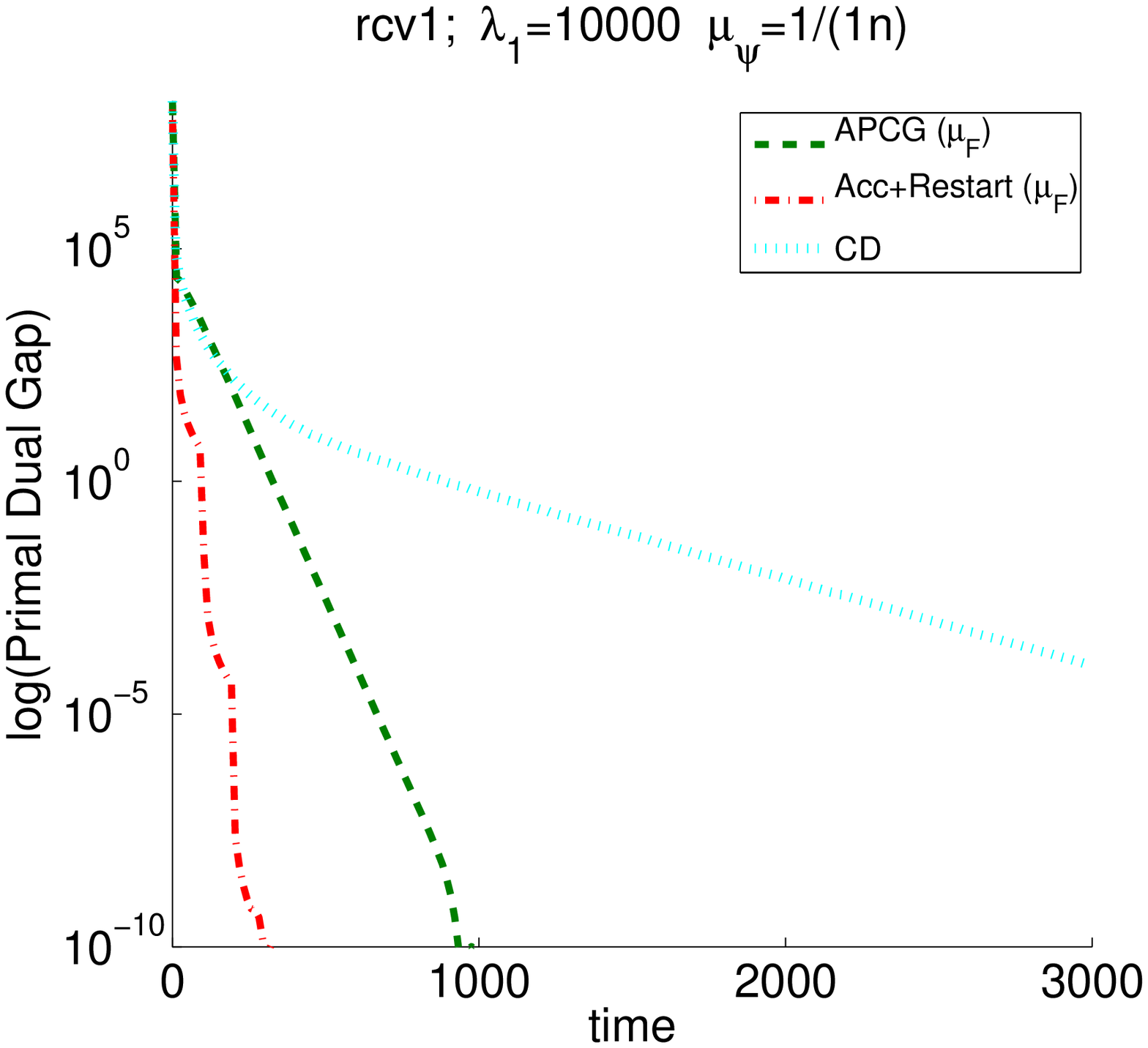}
\includegraphics[width=0.32\linewidth]{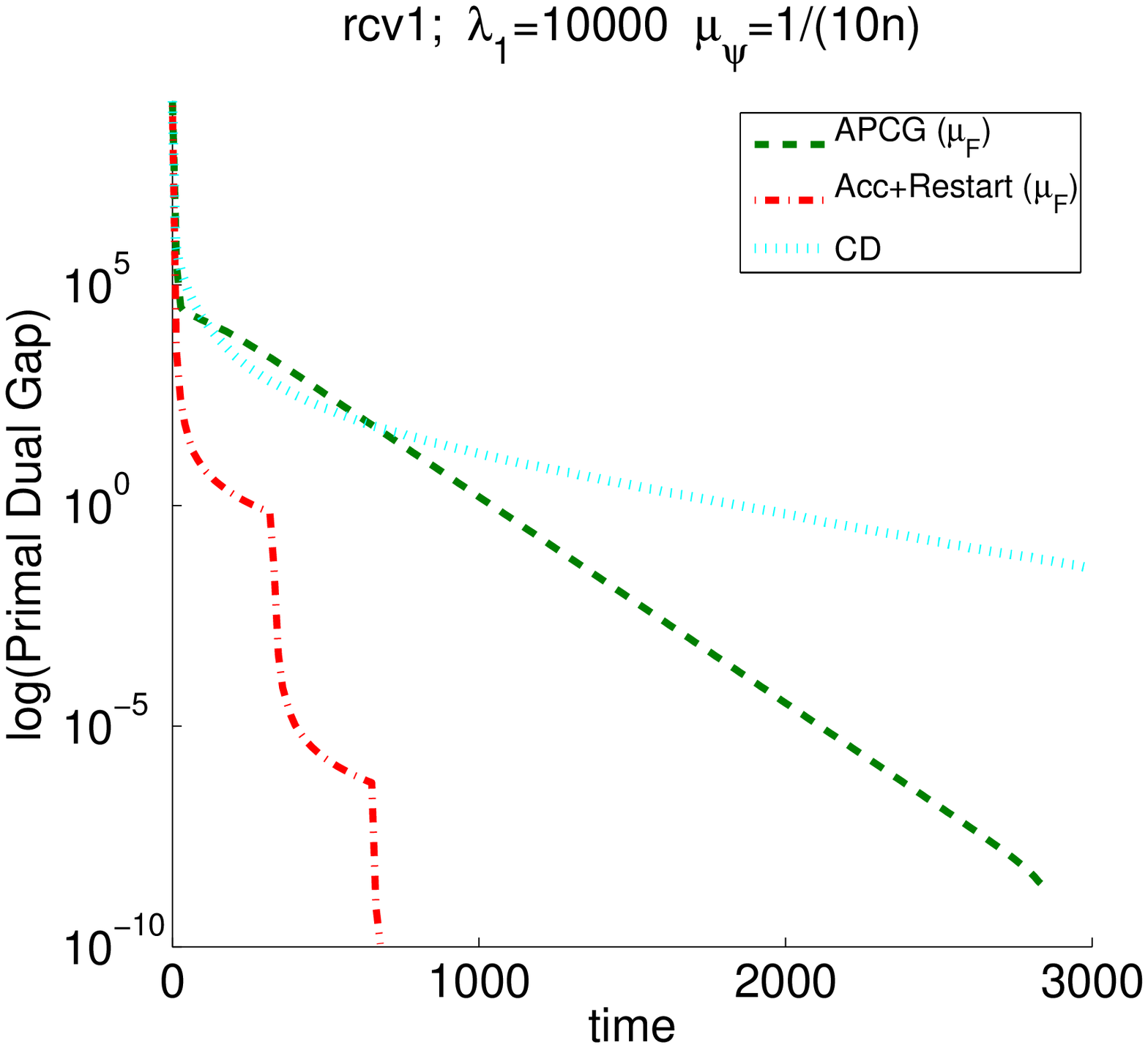}
\includegraphics[width=0.32\linewidth]{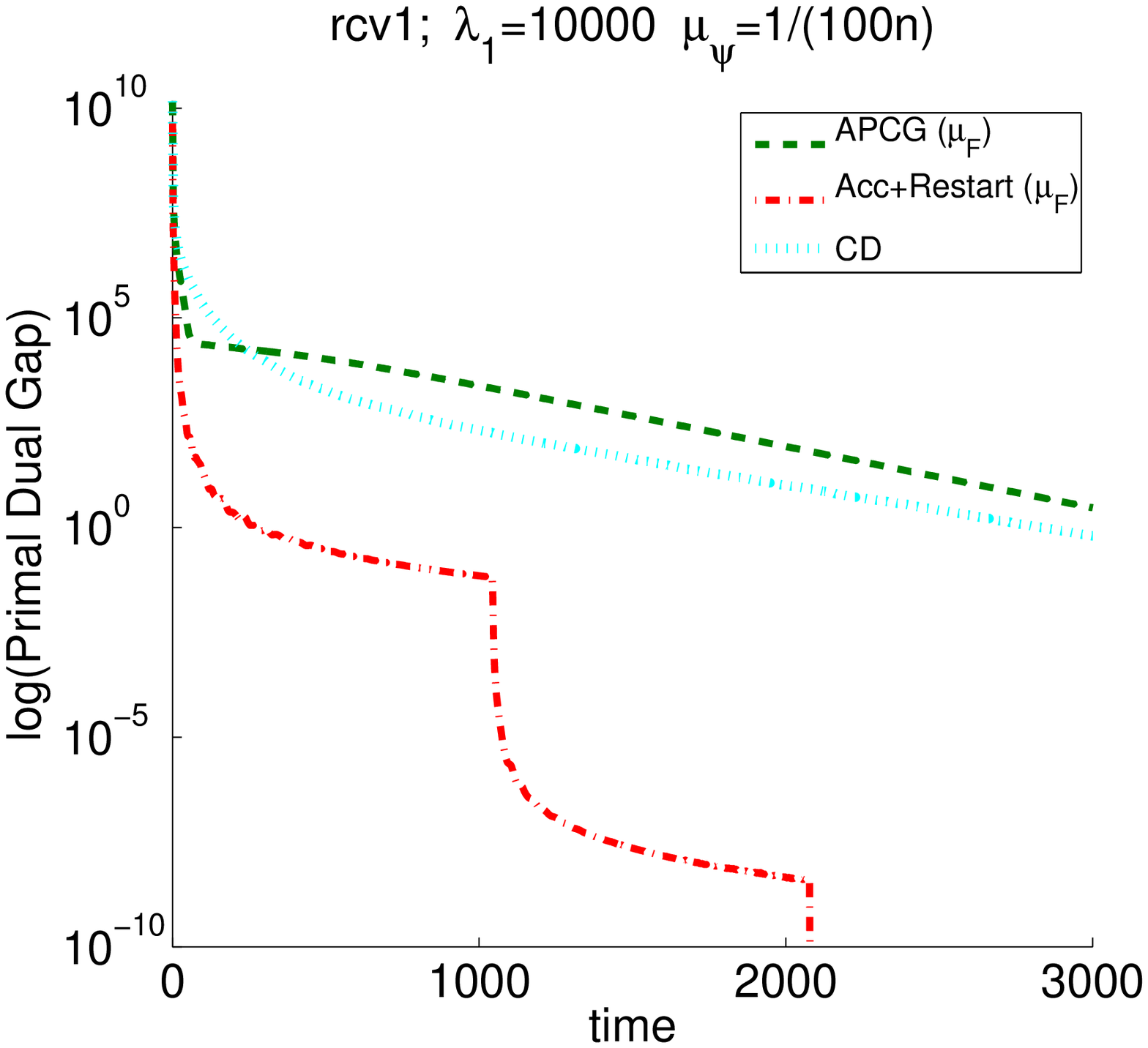}
\includegraphics[width=0.32\linewidth]{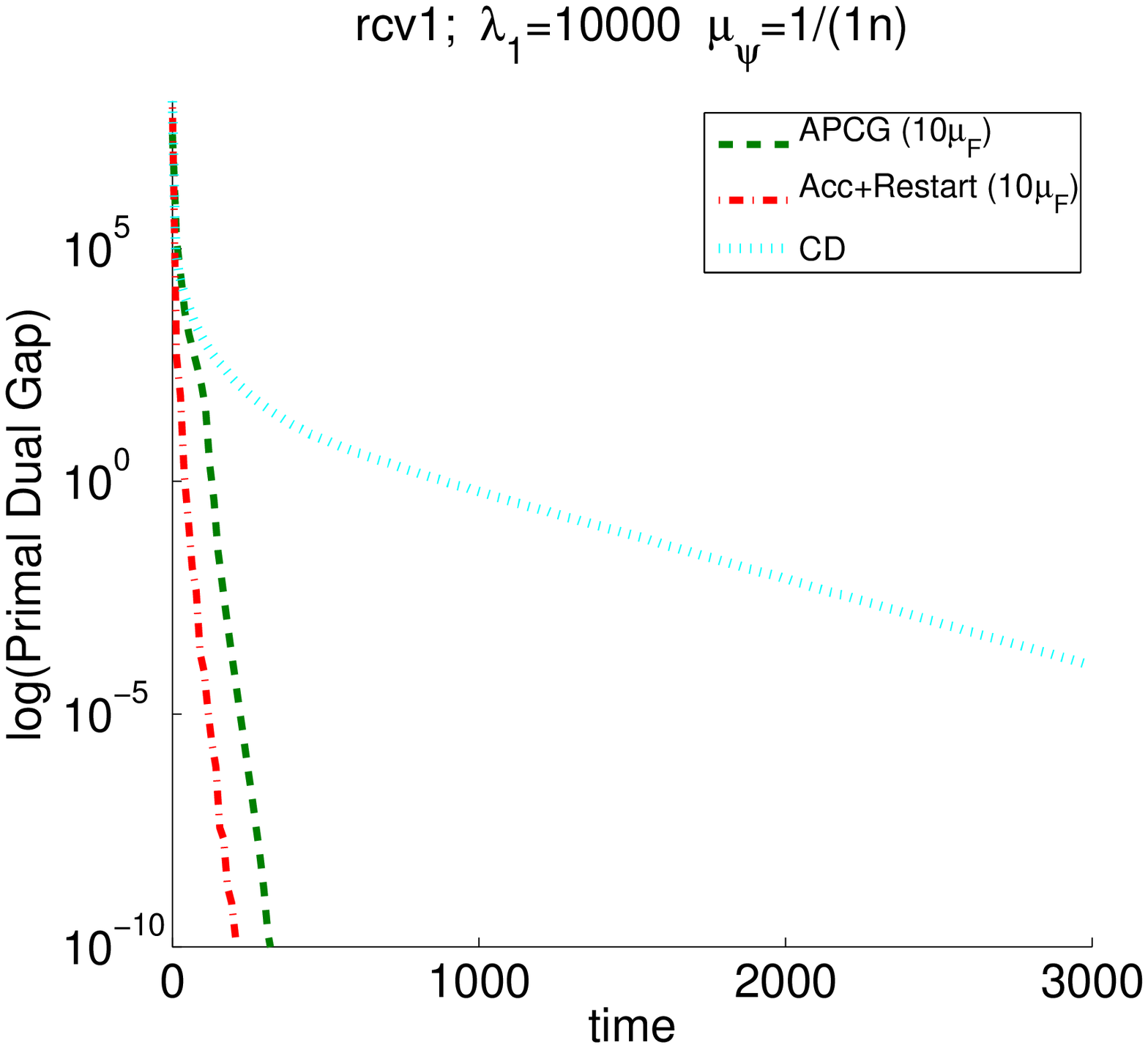}
\includegraphics[width=0.32\linewidth]{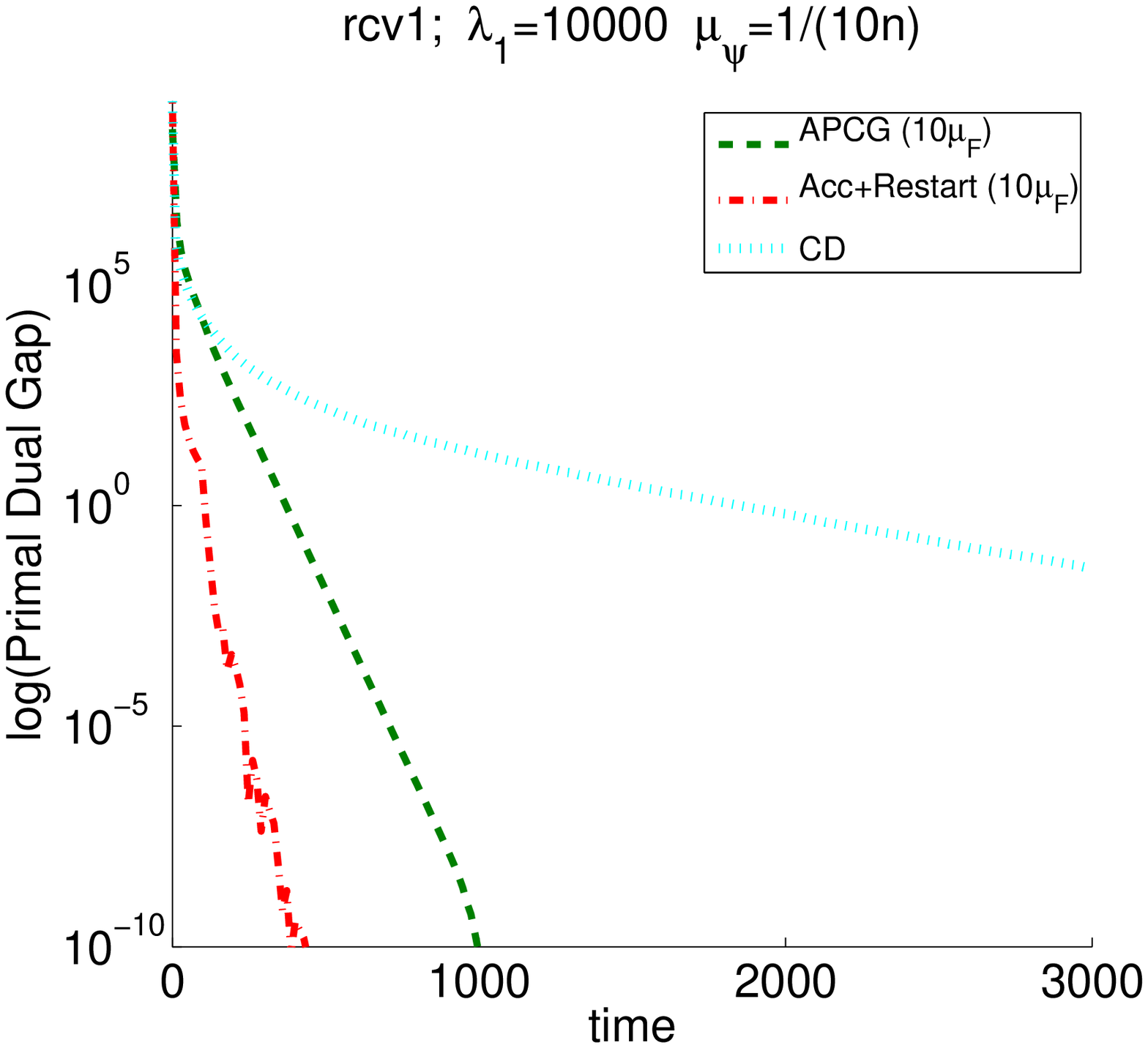}
\includegraphics[width=0.32\linewidth]{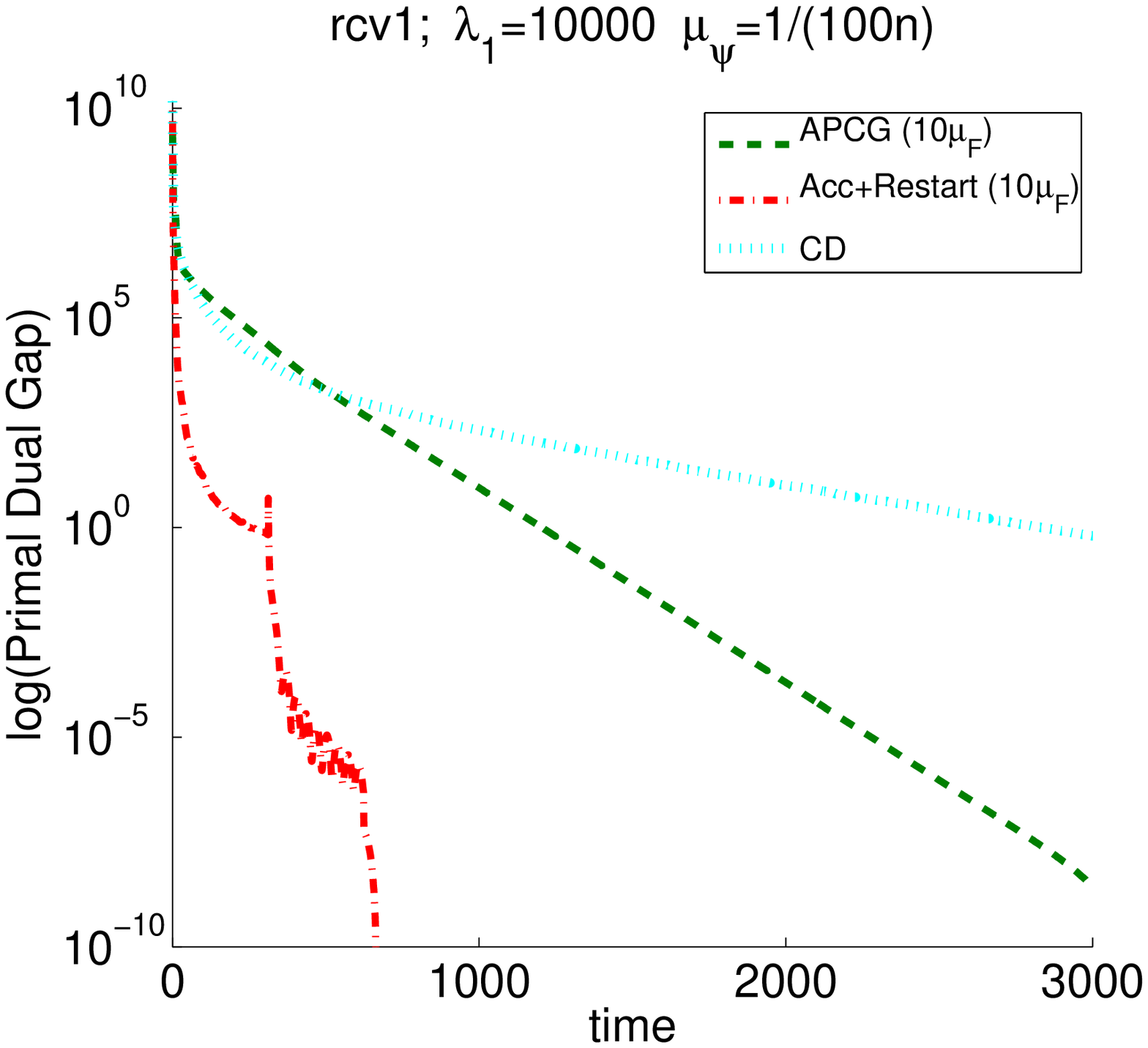}
\includegraphics[width=0.32\linewidth]{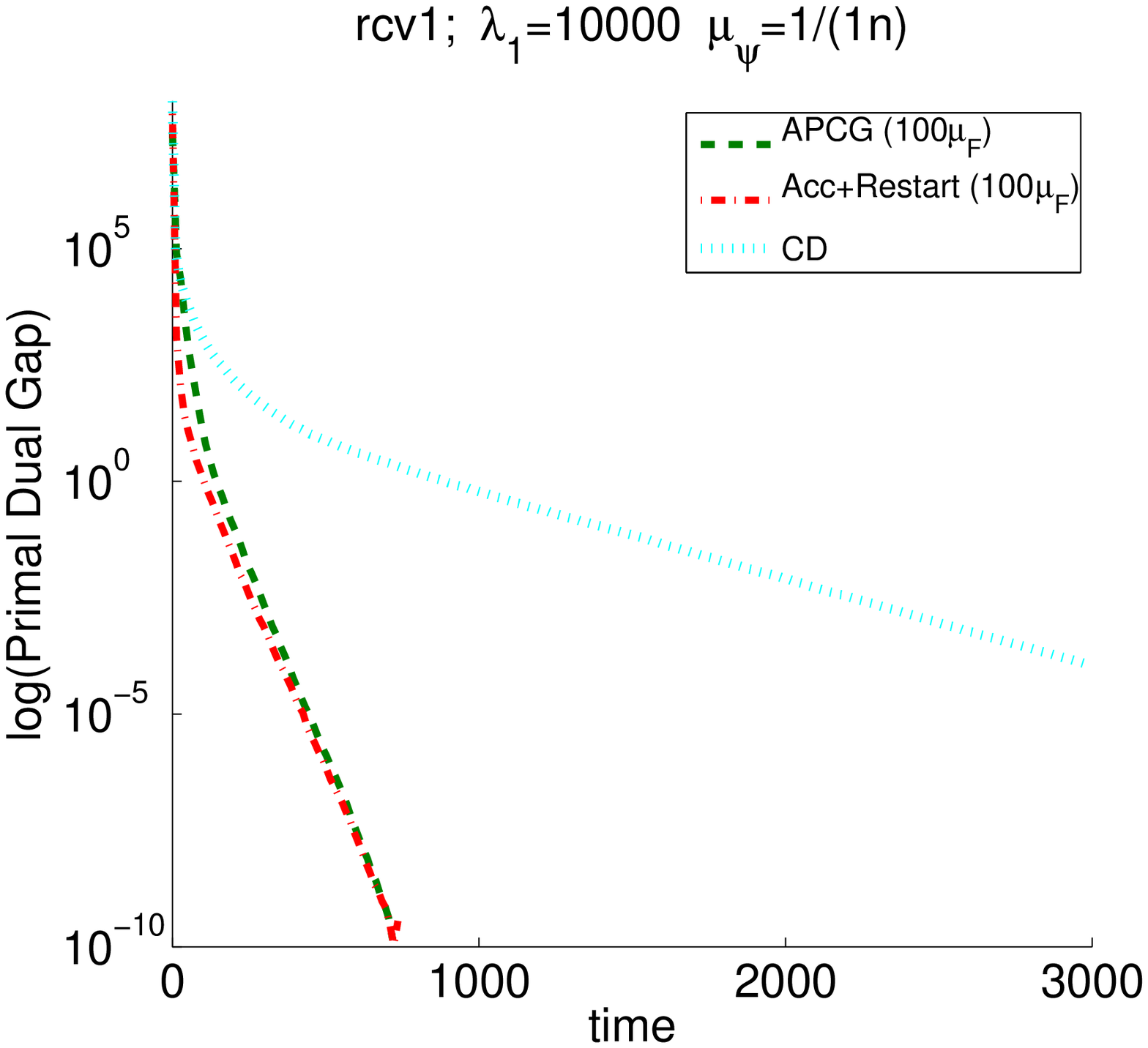}
\includegraphics[width=0.32\linewidth]{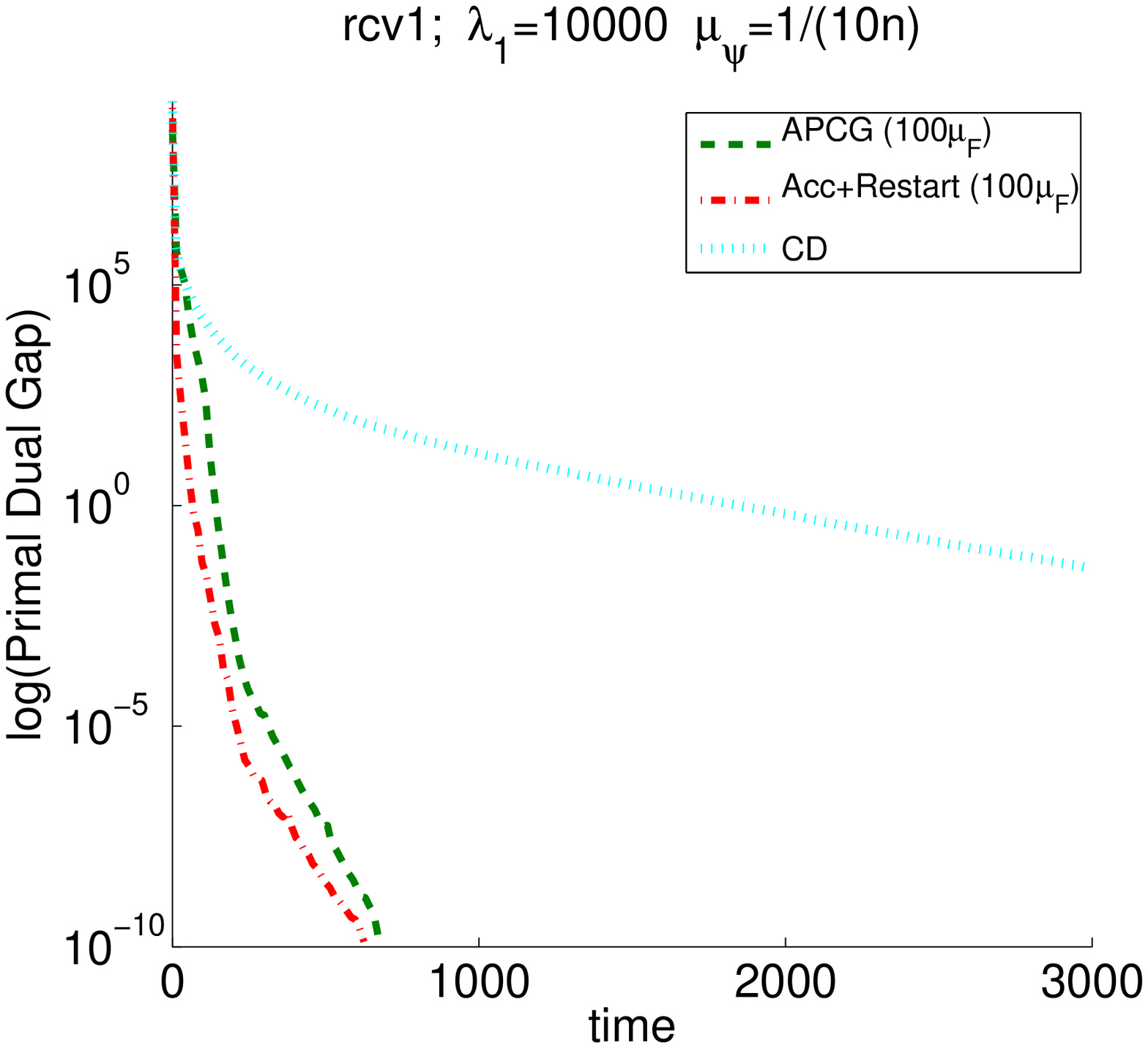}
\includegraphics[width=0.32\linewidth]{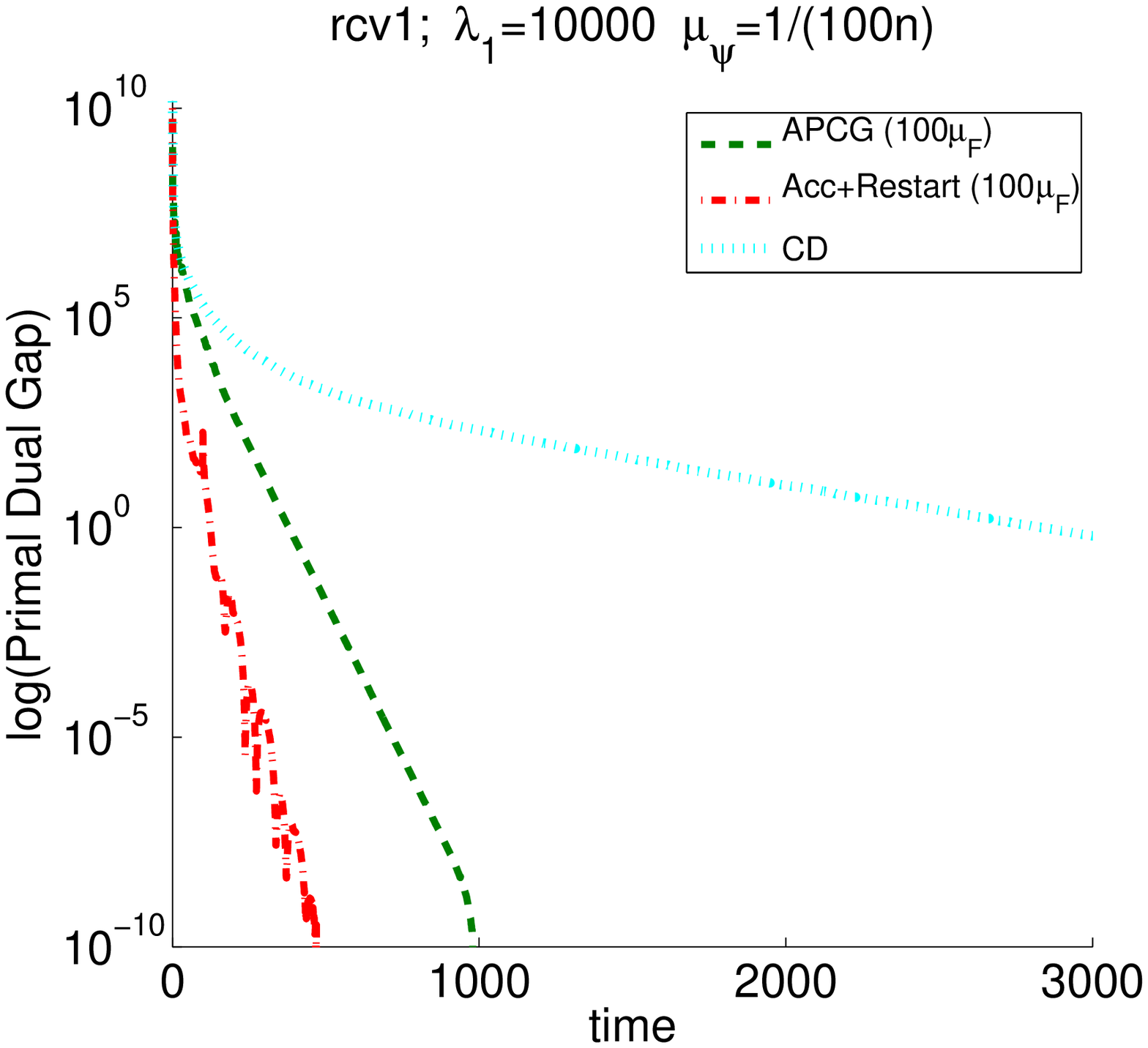}
\includegraphics[width=0.32\linewidth]{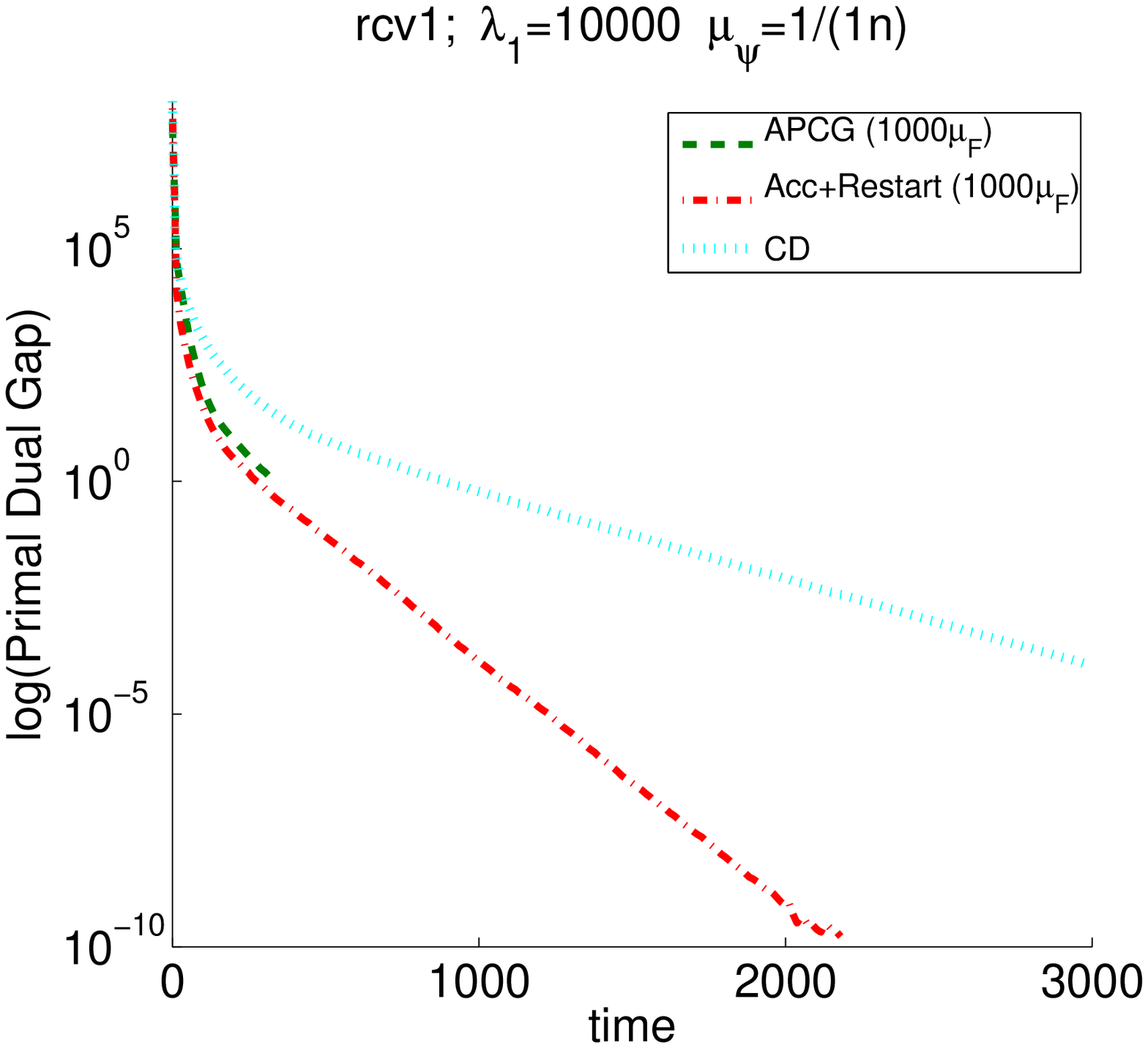}
\includegraphics[width=0.32\linewidth]{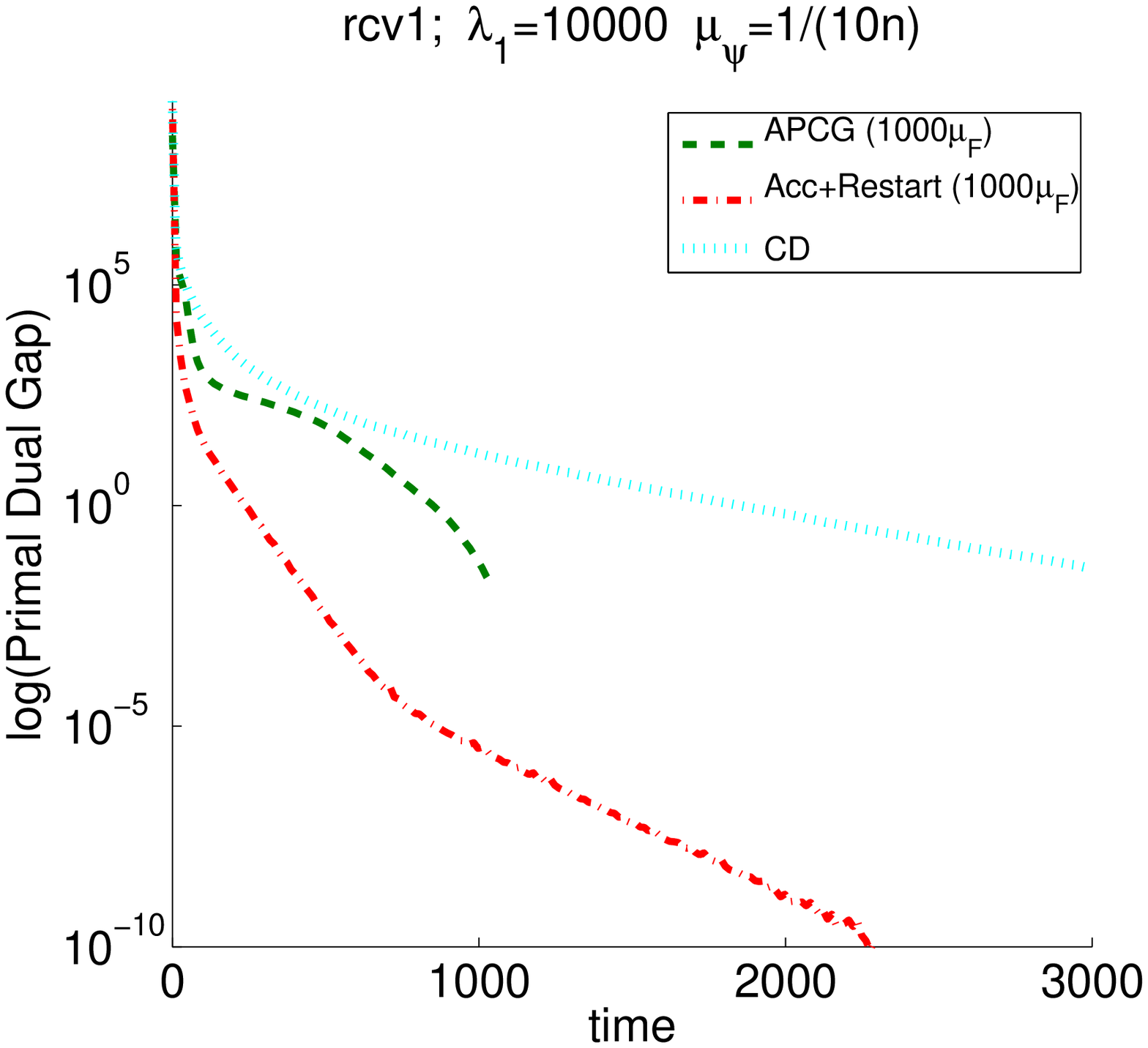}
\includegraphics[width=0.32\linewidth]{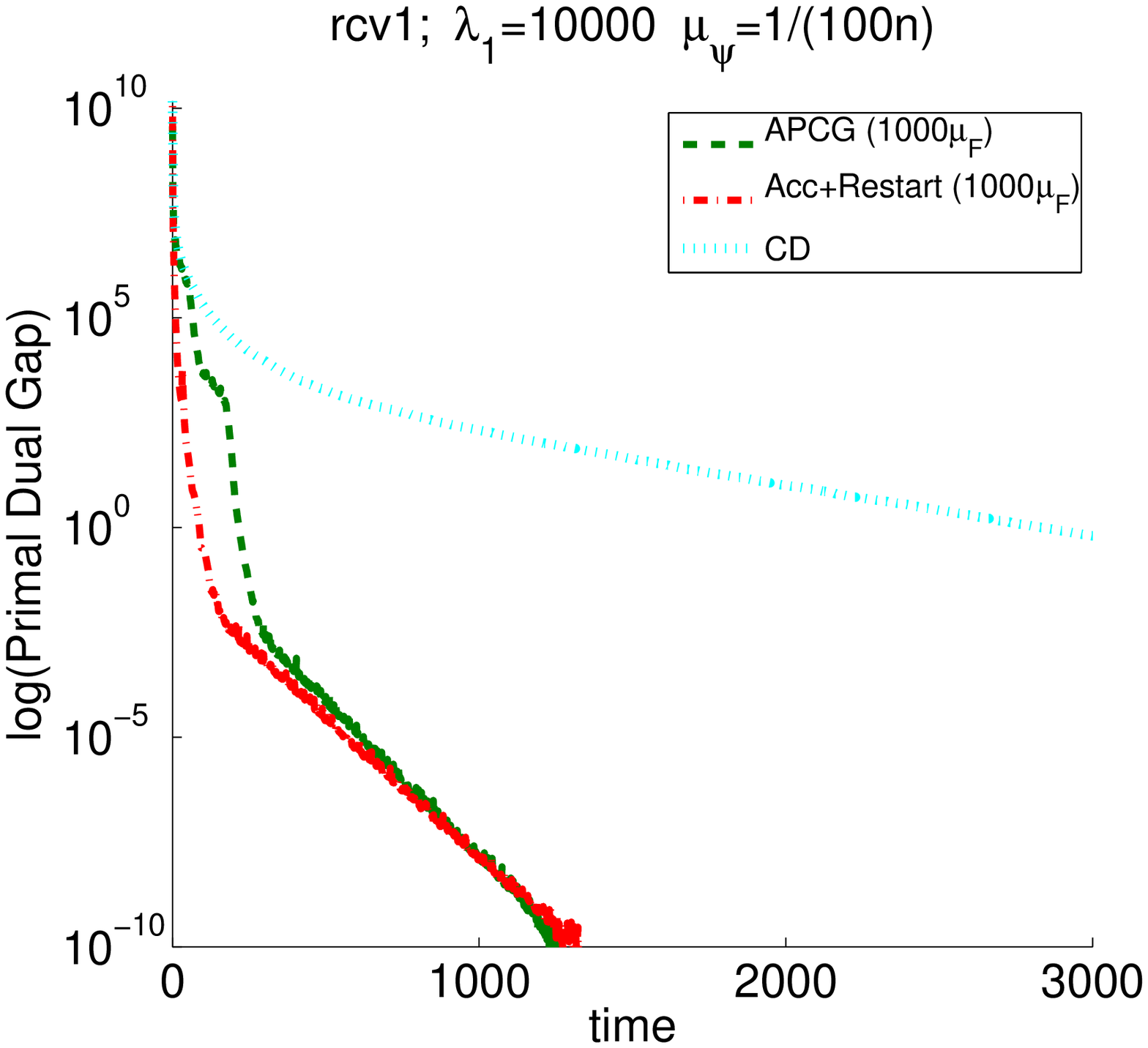}
\caption{Plots of results for dataset RCV1; APCG and APPROX-restart run with: $\mu=\mu_F(v)$ (first row), $\mu = 10\mu_F(v)$ (second row), $\mu = 100 \mu_F(v)$ (third row) and $\mu = 1000 \mu_F(v)$ (fourth row). Note that each column corresponds to the same problem, only the parametrization of the algorithms differ. APCG failed for $(\mu_\psi=1/n, \mu = 1000 \mu_\psi)$ and $(\mu_\psi=0.1/n, \mu = 1000 \mu_\psi)$.
}
\label{fig:d0}
\end{figure}

\appendix

\section{Efficient implementation of the restart point}
\label{sec:impl}

Computing $\mathring x_k$ using \eqref{eq:center_approx} 
will be inefficient for APPROX because it involves full-dimensional
operations. We show in the following that a reformulation similar 
to the one in~\cite{lee2013efficient}  can be done to avoid such expensive 
computations.  

Define:
\begin{align}\label{a:wkdef}
w_k=\theta_{k-1}^{-2}(x_k-z_k), \enspace \forall k\geq 1
\end{align}
We first recall from~\cite{FR:2013approx} that the update for $\{w_{k}\}$ can  be as efficient as for $\{z_{k}\}$:
$$
w_{k+1}=w_{k}-\frac{1-\frac{n}{\tau}\theta_k}{\theta_k^2}(z_{k+1}-z_k).
$$
Define:
\begin{align}\label{a:alphabetadef}
&\alpha_i:=\frac{\gamma_{i+1}^i}{\theta_i^2 \theta_{i-1}^2},\enspace \beta_i :=\frac{\gamma_{i+1}^i}{\theta_i^2 },\enspace i=0,1,\dots 
\end{align}
 It is easy to deduce from the recursive equation~\eqref{arectheta} that
\begin{align}\label{a:gammakplusone}
\gamma_{k+1}^i=\frac{\theta_k^2\gamma_{i+1}^i}{\theta_{i}^2},\enspace \forall i=0,\dots,k.
\end{align}
Consequently,
\begin{align}\label{a:sumvec}&\sum_{i=0}^{k}\frac{\gamma_{k+1}^i}{\theta_{i-1}^2}x_i=
\sum_{i=0}^{k}\frac{\theta_k^2\gamma_{i+1}^i}{\theta_{i}^2\theta_{i-1}^2} x_i=
\theta_k^2\sum_{i=0}^k \alpha_i x_i.
\end{align}
Next we show that the sum vector~\eqref{a:sumvec} can be obtained efficiently using auxiliary coefficients and vectors. Define:
\begin{align}
& a_{i+1}:=\sum_{j=0}^i \alpha_j,\enspace b_{i+1}:=\sum_{j=0}^i  \beta_j,\enspace i=0,1,\dots
\\ & 
g_{i+1}:=\sum_{j=0}^i a_{j+1} (z_{j+1}-z_j), \enspace h_{i+1}:=\sum_{j=0}^i b_{j+1} (w_{j+1}-w_j),\enspace
i=0,1,\dots
\end{align}
Then clearly the update for $\{g_k\}$ and $\{h_k\}$ are also as efficient as for $\{z_{k}\}$:
$$
g_{k+1}=g_k+a_{k+1}(z_{k+1}-z_k),\enspace h_{k+1}=h_k-\frac{b_{k+1}(1-\frac{n}{\tau}\theta_k)}{\theta_k^2}(z_{k+1}-z_k).
$$
Moreover, it is easy to see that
\begin{align} \notag
\sum_{i=0}^k \alpha_i x_i &
\overset{\eqref{a:wkdef}}{=}\sum_{i=0}^k \alpha_i (z_i+\theta_{i-1}^2 w_i )
\overset{\eqref{a:alphabetadef}}{=}\sum_{i=0}^k \alpha_i z_i+\beta_i w_i\\ \notag
&=-\sum_{i=0}^k a_{i+1} (z_{i+1}-z_i)-\sum_{i=0}^k b_{i+1} (w_{i+1}-w_i)+a_{k+1}z_{k+1}+b_{k+1} w_{k+1}
\\&=-g_{k+1}-h_{k+1}+a_{k+1}z_{k+1}+b_{k+1} w_{k+1}\label{a:sumalphx}
\end{align}
 Hence,
\begin{align*}
\bar x_{k+1}&=\sigma x_{k+1}+(1-\sigma)\mathring x_{k+1} 
\\ &=\sigma x_{k+1}+ \frac{1-\sigma}{\sum_{i=0}^{k}\frac{\gamma_{k+1}^i}{\theta_{i-1}^2}+ \frac{n}{\tau}(\frac{1}{\theta_{k}} - \frac{n}{\tau}+1)}\Big[ \sum_{i=0}^{k}\frac{\gamma_{k+1}^i}{\theta_{i-1}^2}x_i+ \frac{n}{\tau}(\frac{1}{\theta_{k} }- \frac{n}{\tau}+1)x_{k+1} \Big]
\\&\overset{\eqref{a:gammakplusone}}{=}\sigma x_{k+1}+ \frac{1-\sigma}{\theta_k^2\sum_{i=0}^{k}\alpha_i+ \frac{n}{\tau}(\frac{1}{\theta_{k}} - \frac{n}{\tau}+1)}\Big[ \theta_k^2\sum_{i=0}^{k}\alpha_i x_i+ \frac{n}{\tau}(\frac{1}{\theta_{k} }- \frac{n}{\tau}+1)x_{k+1} \Big]
\\ &= \sigma x_{k+1}+ \frac{1-\sigma}{\theta_k^2a_{k+1}+ \frac{n}{\tau}(\frac{1}{\theta_{k}} - \frac{n}{\tau}+1)}\Big[ \theta_k^2\sum_{i=0}^{k}\alpha_i x_i+ \frac{n}{\tau}(\frac{1}{\theta_{k} }- \frac{n}{\tau}+1)x_{k+1} \Big]\\
&=\frac{\left(\sigma\theta_k^2 a_{k+1}+\frac{n}{\tau}(\frac{1}{\theta_{k}} - \frac{n}{\tau}+1)\right)x_{k+1}+(1-\sigma)\theta_k^2\sum_{i=0}^{k}\alpha_i x_i}{\theta_k^2a_{k+1}+ \frac{n}{\tau}(\frac{1}{\theta_{k}} - \frac{n}{\tau}+1)}
\end{align*}
Finally we plug in~\eqref{a:wkdef} and~\eqref{a:sumalphx} to obtain:
\begin{align*}
\bar x_{k+1}= z_{k+1}+\theta_k^2 w_{k+1}+\frac{(1-\sigma)\theta_k^2(-g_{k+1}-h_{k+1})+(1-\sigma)(\theta_k^2 b_{k+1}- \theta_k^4 a_{k+1}) w_{k+1}}{\theta_k^2 a_{k+1}+ \frac{n}{\tau}(\frac{1}{\theta_{k}} - \frac{n}{\tau}+1)}.
\end{align*}
The above reasoning showed that Algorithm~\ref{alg:approxree} is equivalent to Algorithm~\ref{algo:restartedFGMv2}. More importantly, note that full dimensional operations are avoided in Algorithm~\ref{alg:approxree}.
\begin{algorithm}[H]
\begin{algorithmic}[1]
\STATE{\textbf{Parameters}: Choose $\hat S$, $x_0\in \R^n$, $K\in \bN$, $\sigma \in (0,1)$.} 
\STATE{\textbf{Initialization}: Set $\tau=\Exp[|\hat S|]$, $\theta_0= \frac{\tau}{n}$, $r_0=0$, $a_{0}=0$, $b_{0}=0$ and $g_{0}=0$, $h_{0}=0$, $z_0=x_0$, $u_0=0$. }
\FOR{$k\geq 0$} 
\STATE $w_{k+1}\leftarrow w_k$\\
\STATE $z_{k+1}\leftarrow z_k$\\
\STATE  $a_{k+1} = a_{k}+\frac{r_k(1-\theta_k)}{\theta_k^4}$
\STATE $b_{k+1} = b_{k}+\frac{r_k}{\theta_k^2 }$
\STATE Randomly generate $S_k \sim \hat{S}$\\
\FOR{ $i \in S_k$}
\STATE
$
t_{k}^i=\arg\min_{t\in \R} \big\{\< \nabla_i f(\theta_{k}^2w_k+z_k), t>+\frac{n\theta_k v_i}{2\tau} |t|^2+\psi^i(z_k^i+t) \big\}
$ 
\STATE
$z_{k+1}^i= z_k^i+t_k^i$
\STATE
$ w_{k+1}^i= w_k^i-\frac{1-\frac{n}{\tau}\theta_k}{\theta_{k}^2} t_k^i$ \label{line:update_u}
\STATE 
$g_{k+1}^i=g_{k}^i+a_{k+1}t_k^i$
\STATE 
$h_{k+1}^i=h_{k}^i-\frac{b_{k+1}(1-\frac{n}{\tau}\theta_k)}{\theta_{k}^2} t_k^i$
\ENDFOR
\IF{ $k \equiv 0 \; \mathrm{mod} \; 
%\left \lfloor %\frac{2ne}{\tau}(1-\frac{\tau}{n}+\frac{1}{\mu_F(v)})^{1/2}-\frac{n-\tau}{\tau} %\right \rfloor 
K
$}
\STATE  $z_{k+1} \leftarrow z_{k+1}+\theta_k^2 w_{k+1}+\frac{(1-\sigma)\theta_k^2(-g_{k+1}-h_{k+1})+(1-\sigma)(\theta_k^2 b_{k+1}- \theta_k^4 a_{k+1}) w_{k+1}}{\theta_k^2 a_{k+1}+ \frac{n}{\tau}(\frac{1}{\theta_{k}} - \frac{n}{\tau}+1)}$
\STATE $w_{k+1}  \leftarrow 0$
\STATE $g_{k+1} \leftarrow 0$
\STATE $h_{k+1} \leftarrow 0$
\STATE $\theta_{k+1}= \theta_0$, $r_{k+1}=0$, $a_{k+1}=0$, $b_{k+1}=0$
\ELSE{}
\STATE $\theta_{k+1}= \frac{\sqrt{\theta_k^4+4\theta_k^2}-\theta_k^2}{2}$
\STATE $r_{k+1}=\theta_{k+1} (1-\frac{n}{\tau}\theta_{k})+\frac{n}{\tau}(\theta_{k}-\theta_{k+1})$
\ENDIF
 \ENDFOR
\STATE $\mathbf{OUTPUT}: \theta_{k}^2 w_{k+1}+z_{k+1}$
\end{algorithmic}
\caption{APPROX restart efficient equivalent }
\label{alg:approxree}
\end{algorithm}

\section{Additional insight on the rate of convergence}
\label{sec:estimation}

%
%For all $l < k$
%\begin{align*}
%\gamma_k^l &= \left(\prod_{j=l+1}^{k-1} (1-\theta_j)\right) \gamma_{k-1}^l
%= \left( \prod_{j=l+1}^{k-1} \frac{\theta_j^2}{\theta_{j-1}^2}\right) \gamma_{l+1}^l = \frac{\theta_{k-1}^2}{\theta_l^2}\left(\theta_l + \frac{n}{\tau} (\theta_{l-1} - \theta_l - \theta_l \theta_{l-1}) \right)
%\\
%&= \frac{\theta_{k-1}^2}{\theta_l^2}\left(\theta_l - \frac{n}{\tau} \theta_l^2(\frac{1}{\theta_l} - \frac{1}{\theta_{l-1}}) \right)
%= \frac{\theta_{k-1}^2}{\theta_l^2}\left(\theta_l - \frac{n}{\tau} \theta_l^2\frac{\frac{1}{\theta_l^2} - \frac{1}{\theta_{l-1}^2}}{\frac{1}{\theta_l} + \frac{1}{\theta_{l-1}}} \right)
%= \frac{\theta_{k-1}^2}{\theta_l^2}\left(\theta_l - \frac{n}{\tau} \theta_l^2\frac{\theta_{l-1}}{\theta_l + \theta_{l-1}} \right)
%\end{align*}
%Hence, 
%\begin{align*}
%\sum_{i=0}^{k-1} \frac{\gamma_k^i}{\theta_{i-1}^2}
% = \theta_{k-1}^2 \sum_{i=0}^{k-1} \frac{1}{\theta_{i-1}^2 \theta_i} - \frac{n}{\tau} \frac{1}{\theta_{i-1}^2}(\frac{\theta_{i-1}}{\theta_i + \theta_{i-1}}) 
%\end{align*}

Define:
$$
\xi_k:=\sum_{i=0}^k \frac{\gamma_k^i}{\theta_{i-1}^2},\enspace k=1,2,\dots.
$$
Then  \begin{align}\label{amKxiK} m_k(\mu)=\frac{\mu\theta_0^2}{1+\mu(1-\theta_0)}\left(\xi_k-\frac{1-\theta_0}{\theta_0^2}\right),\enspace k=1,2,\dots.\end{align}
We first prove a simple recursive equation.
\begin{lemma}
For any $k\geq 1$ we have:
\begin{align}\label{axik}
\xi_{k+1}=(1-\theta_k) \xi_k+ \frac{1+(\frac{n}{\tau}-1)\theta_k}{\theta_k}.
\end{align}
\end{lemma}
\begin{proof}
Let any $k\geq 1$. We first decompose the sum and use~\eqref{eq:gammas} to obtain:
$$
\xi_{k+1}=\sum_{i=0}^{k+1} \frac{\gamma_{k+1}^i}{\theta_{i-1}^2}=\sum_{i=0}^{k-1}
\frac{\gamma_{k+1}^i}{\theta_{i-1}^2}+\frac{\gamma_{k+1}^k}{\theta_{k-1}^2}
+\frac{\gamma_{k+1}^{k+1}}{\theta_k^2}=
(1-\theta_k)\sum_{i=0}^{k-1}
\frac{\gamma_{k}^i}{\theta_{i-1}^2}+\frac{\gamma_{k+1}^k}{\theta_{k-1}^2}
+\frac{\gamma_{k+1}^{k+1}}{\theta_k^2}.
$$Then we get the recursive equation:
$$
\xi_{k+1}=(1-\theta_k)\xi_k+\frac{\gamma_{k+1}^k-(1-\theta_k)\gamma_k^k}{\theta_{k-1}^2}
+\frac{\gamma_{k+1}^{k+1}}{\theta_k^2},
$$
which together with~\eqref{arectheta} and~\eqref{eq:gammas} yields~\eqref{axik}.
\end{proof}

\begin{lemma}
For any $k\geq 1$ we have
\begin{align}\label{athetakmuk}
\frac{1}{3\theta_{k}^2}\leq\xi_k \leq \frac{1}{\theta_k^2}.
\end{align}
\end{lemma}
\begin{proof}
We proceed by induction on $k$. First for $k=1$ we have:
$$
\xi_{1}=\frac{1}{\theta_0^2}\overset{\eqref{arectheta}}{=}\frac{1-\theta_1}{\theta_1^2}\leq \frac{1}{\theta_1^2}.
$$
Now suppose that we have
$$
\xi_k \leq \frac{1}{\theta_k^2}
$$
for some $k\geq 1$.  Then
\begin{align*}
\xi_{k+1} \overset{\eqref{axik}}{\leq}&
\frac{1-\theta_k}{\theta_k^2}+ \frac{1+(\frac{n}{\tau}-1)\theta_k}{\theta_k}
\leq \frac{1}{\theta_k^2}+\frac{1}{\theta_k}\overset{\eqref{arectheta}}{=}\frac{1-\theta_{k+1}}{\theta_{k+1}^2}+\frac{1}{\theta_k}=\frac{1}{\theta_{k+1}^2}-\frac{1}{\theta_{k+1}}+\frac{1}{\theta_k} \\
\overset{\eqref{atheradecr}}{\leq}&\frac{1}{\theta_{k+1}^2}.
\end{align*}
Thus we proved by recurrence the following upper bound:
$$
\xi_k \leq \frac{1}{\theta_k^2},\enspace \forall k=1,2,\dots.
$$
Next we prove the lower bound again by induction on $k$. Since $\theta_k\leq 1$, 
we first observe that
$$
\theta_{k+1}=\frac{\sqrt{\theta_k^4+4\theta_k^2}-\theta_k^2}{2}
\leq  \frac{\sqrt{5}-1}{2}\leq \frac{2}{3},\enspace \forall k=0,1,2,\dots
$$
Then we get:
\begin{align}\label{athe4}
\frac{1}{\theta_{k}^2}=\frac{1-\theta_{k+1}}{\theta_{k+1}^2}\geq \frac{1}{3\theta_{k+1}^2},\enspace \forall k=0,1,2,\dots
\end{align}
In particular,
$$
\xi_1=\frac{1}{\theta_0^2}
\geq \frac{1}{3\theta_1^2}.
$$
Now suppose that we have
$$
\xi_k \geq \frac{1}{3\theta_k^2}
$$
for some $k\geq 1$. Then
$$
\xi_{k+1} \overset{\eqref{axik}}{\geq}
\frac{1-\theta_k}{3\theta_k^2}+ \frac{1+(\frac{n}{\tau}-1)\theta_k}{\theta_k}
\geq \frac{1}{3\theta_k^2}-\frac{1}{3\theta_k}+\frac{1}{\theta_k}
\overset{\eqref{arectheta}}{=}\frac{1}{3\theta_{k+1}^2}-\frac{1}{3\theta_{k+1}}+\frac{2}{3\theta_k}
$$ 
Next we apply~\eqref{athe4} to obtain:
$$
\xi_{k+1} \geq \frac{1}{3\theta_{k+1}^2}-\frac{1}{3\theta_{k+1}}+\frac{2}{3 \sqrt{3} \theta_{k+1}} \geq \frac{1}{3\theta_{k+1}^2}+\frac{1}{3\theta_{k+1}}-\frac{1}{3\theta_{k+1}}= \frac{1}{3\theta_{k+1}^2}.
$$
We thereby proved the lower bound for any $k$.
\end{proof}
We then deduce directly from~\eqref{athetabd} the following corollary.
%\todo[inline]{I did not check the left-hand side inequality in~\eqref{athetabd}.}
\begin{coro}
For any $k\geq 1$ we have
\begin{align}\label{athetakmuk1}
\frac{(k+2/\theta_0)^2}{12}\leq\xi_k \leq (k+1/\theta_0)^2.
\end{align}
\end{coro}
\iffalse

The rate that we want to minimize is:
$$
\left(\max(\sigma, 1-\sigma m_K(\mu_F(v)))\right)^{1/K}.
$$
We first observe that for fixed $K$, the best $\sigma$ should be such that 
$$
\sigma=1-\sigma m_K(\mu_F(v)).
$$
We plug in the optimal $\sigma$ into the rate and obtain:
$$
\left(\frac{1}{1+m_K(\mu_F(v))}\right)^{1/K}
$$
Then we use~\eqref{amkmu} and~\eqref{athetakmuk1} to further upper bound this term:
\begin{align*}
\left(\frac{1}{1+m_K(\mu_F(v))}\right)^{1/K}
&= \left(\frac{1}{1+\frac{\mu \theta_0^2}{1+\mu(1-\theta_0)}\left(\xi_K-\frac{1-\theta_0}{\theta_0^2}\right)}\right)^{1/K}
\\&=\left(\frac{1+\mu(1-\theta_0)}{1+\mu \theta_0^2 \xi_K}\right)^{1/K}\\\label{afgrtg}
&\overset{\eqref{athetakmuk1}}{\leq} 
\left(\frac{1+\mu(1-\theta_0)}{1+\mu \theta_0^2 (K-1+2/\theta_0)^2/4}\right)^{1/K}
\end{align*}

\fi
\iffalse
Let $$
\lambda=-\frac{2}{W_0(-e^{-2})}-1,$$
where $W_0$ denotes the main branch of the Lambert-W function.
Note that the function
$$
\frac{\log(\frac{1+\lambda}{2})}{\sqrt{\lambda}}
$$
achieves maximum at
$
\lambda$. The approximate value of $\lambda$ is $11.6$.
\fi

\begin{lemma}\label{l:Kexys}
Let   $\lambda \geq \mu$ and \begin{align}\label{aKff}
K=\ceil*{\frac{2\sqrt{3}}{\theta_0}\sqrt{\frac{\lambda}{\mu}}-\frac{2}{\theta_0}}.
\end{align}
Then the following inequalities hold:
\begin{align}
& \lambda\leq \mu \theta_0^2 \xi_K\leq 9 \lambda  \label{afgd1},
\\&  {K} \leq \frac{2\sqrt{3}\sqrt{\lambda}}{\theta_0\sqrt{\mu}}.\label{afgggg}
\end{align}
\end{lemma}
\begin{proof}
Consider the interval 
$$\left [\frac{2\sqrt{3}}{\theta_0}\sqrt{\frac{\lambda}{\mu}}-\frac{2}{\theta_0}, \frac{2\sqrt{3}}{\theta_0}\sqrt{\frac{\lambda}{\mu}}-\frac{1}{\theta_0}
\right] .
$$We first observe that it is included in $\R_{>0}$ because $\lambda \geq \mu$. Moreover, the length of the interval 
is larger than~1. Therefore, $K$ defined by~\eqref{aKff} satisfies:
$$
\frac{2\sqrt{3}}{\theta_0}\sqrt{\frac{\lambda}{\mu}}-\frac{2}{\theta_0} \leq K \leq \frac{2\sqrt{3}}{\theta_0}\sqrt{\frac{\lambda}{\mu}}-\frac{1}{\theta_0}
$$Consequently we obtain~\eqref{afgggg} and
$$
\frac{(K+2/\theta_0)^2}{12}\geq \frac{\lambda}{\mu\theta_0^2}, \enspace
{(K+1/\theta_0)^2}\leq \frac{12 \lambda}{\mu \theta_0^2},
$$
which together with~\eqref{athetakmuk1} implies~\eqref{afgd1}. 
\end{proof}

\begin{proposition}
Let $\lambda \geq \mu$.  Choose \begin{align}\label{aK_appendix}
K=\ceil*{\frac{2\sqrt{3}}{\theta_0}\sqrt{\frac{\lambda}{\mu}}-\frac{2}{\theta_0} },
\end{align}
and 
\begin{align}\label{asigma_appendix}
\sigma=\frac{1}{1+m_K(\mu)}.
\end{align}
Then the iterates of Algorithm~\ref{algo:restartedFGMv2} satisfy for any $k \geq K$
\[
\Exp[\Delta(x_k)] \leq  \left(1- \min \left(\frac{\mu_F(v)}{\mu} ,1\right)\frac{\lambda-\mu(1-\theta_0)}{\lambda+1} \right) ^{\frac{k\theta_0 \sqrt{\mu}}{2\sqrt{3}\sqrt{\lambda}}-1}\Delta(x_0).
\]
\label{prop:choose_sigma_and_K_appendix}
\end{proposition}
\begin{proof}
Let $\{x_k\}$ be the iterates of Algorithm~\ref{algo:restartedFGMv2} using $\sigma$ and $K$ defined by~\eqref{aK_appendix} and~\eqref{asigma_appendix}. By Proposition~\ref{th:sigmaK},  for any $k \geq K$,
\[
\Exp[\Delta(x_k)] \leq   \left(\max\left( \sigma, 1 -\sigma m_K(\mu_F(v)) \right)^{1/K}\right)^{k-K}\Delta(x_0).
\]
Thus we just need to prove:
$$
\left(\max\left( \sigma, 1 -\sigma m_K(\mu_F(v)) \right)^{1/K}\right)^{k-K} 
\leq \left(1- \min \left(\frac{\mu_F(v)}{\mu} ,1\right)\frac{\lambda-\mu(1-\theta_0)}{\lambda+1} \right) ^{\frac{k\theta_0 \sqrt{\mu}}{2\sqrt{3}\sqrt{\lambda}}-1}.
$$
We first note that if $\mu>\mu_F(v)$, then 
$$
m_K(\mu) >m_K(\mu_F(v)).
$$
Therefore,
\begin{align*}
&\max\left( \sigma, 1 -\sigma m_K(\mu_F(v)) \right)\\&=1_{\mu\leq \mu_F(v)}\frac{1}{1+m_K(\mu)}+ 1_{\mu>\mu_F(v)}\frac{1+m_K(\mu)-m_K(\mu_F(v))}{1+m_K(\mu)}
\end{align*}
where 
$$
1_{\mu\leq \mu_F(v)}=\left\{\begin{array}{ll}1 & \mathrm{if~} \mu \leq \mu_F(v) \\ 0 & \mathrm{otherwise} \end{array}\right. \enspace,
$$
and $ 1_{\mu> \mu_F(v)}=1- 1_{\mu\leq \mu_F(v)}$. Next we replace $m_K(\mu)$ using~\eqref{amKxiK} and rearrange the terms:
\begin{align*}\max&\Big( \sigma, 1 -\sigma m_K(\mu_F(v)) \Big)\\
% % % % % % % % % % % % %
&=1_{\mu\leq \mu_F(v)}\frac{1}{1+\frac{\mu \theta_0^2}{1+\mu(1-\theta_0)}\left(\xi_K-\frac{1-\theta_0}{\theta_0^2}\right)} +1_{\mu>\mu_F(v)}
\frac{1+ \left(\frac{\mu \theta_0^2}{1+\mu(1-\theta_0)}-\frac{\mu_F(v) \theta_0^2}{1+\mu_F(v)(1-\theta_0)}\right)\left(\xi_K-\frac{1-\theta_0}{\theta_0^2}\right)}{1+\frac{\mu \theta_0^2}{1+\mu(1-\theta_0)}\left(\xi_K-\frac{1-\theta_0}{\theta_0^2}\right)}
% % % % % % % % % % % %
\\&=1_{\mu\leq \mu_F(v)}
\frac{1+\mu(1-\theta_0)}{1+\mu \theta_0^2 \xi_K} +1_{\mu> \mu_F(v)} \frac{1+\mu(1-\theta_0)+\left(
\mu \theta_0^2 -\mu_F(v)\theta_0^2 \frac{1+\mu(1-\theta_0)}{1+\mu_F(v)(1-\theta_0)}
\right)\left(\xi_K-\frac{1-\theta_0}{\theta_0^2}\right)}{1+\mu \theta_0^2 \xi_K}
% % % % % % % % % % % %
\\&=1_{\mu\leq \mu_F(v)}
\frac{1+\mu(1-\theta_0)}{1+\mu \theta_0^2 \xi_K} 
1_{\mu> \mu_F(v)} \frac{1+\mu\theta_0^2\xi_K-\left( \frac{1+\mu(1-\theta_0)}{1+\mu_F(v)(1-\theta_0)}
\right)\mu_F(v) \theta_0^2\left(\xi_K-\frac{1-\theta_0}{\theta_0^2}\right)}{1+\mu \theta_0^2 \xi_K}
\end{align*}
By~\eqref{afgd1},
$$
\xi_K \geq \frac{\lambda}{\mu \theta_0^2} \geq \frac{1-\theta_0}{\theta_0^2}.
$$ 
Therefore, 
\begin{align}\notag
&\max\left( \sigma, 1 -\sigma m_K(\mu_F(v)) \right)
\\ \notag&\leq 1_{\mu\leq \mu_F(v)}
\frac{1+\mu(1-\theta_0)}{1+\mu \theta_0^2 \xi_K}+1_{\mu> \mu_F(v)} \frac{1+\mu\theta_0^2\xi_K-\mu_F(v) \theta_0^2\left(\xi_K-\frac{1-\theta_0}{\theta_0^2}\right)}{1+\mu \theta_0^2 \xi_K}
\\\notag &=
 1_{\mu\leq \mu_F(v)} \left(
1-\frac{\mu\theta_0^2(\xi_K-\frac{1-\theta_0}{\theta_0^2})}{1+\mu \theta_0^2 \xi_K}
\right)+1_{\mu> \mu_F(v)} \left(
1-\frac{\mu_F(v)\theta_0^2(\xi_K-\frac{1-\theta_0}{\theta_0^2})}{1+\mu \theta_0^2 \xi_K}
\right)
\\\notag &=1-\min\left(1,\frac{\mu_F(v)}{\mu}\right)\frac{\mu\theta_0^2\xi_K-\mu(1-\theta_0)}{1+\mu \theta_0^2 \xi_K}
\\&
\leq 1-\min\left(1,\frac{\mu_F(v)}{\mu}\right)\frac{\lambda-\mu(1-\theta_0)}{\lambda+1}.\label{arggffg}
\end{align}
Consequently,
\begin{align}\label{a:boundmax}
\left(\max\left( \sigma, 1 -\sigma m_K(\mu_F(v)) \right)^{1/K}\right)^{k-K} 
\leq  \left(1-\min\left(1,\frac{\mu_F(v)}{\mu}\right)\frac{\lambda-\mu(1-\theta_0)}{\lambda+1}\right)^{k/K-1}
\end{align}
Next we apply~\eqref{afgggg} and get:
\begin{align}\label{afgert}
   \left( 1-\min \Big( 1,\frac{\mu_F(v)}{\mu} \Big)\frac{\lambda-\mu(1-\theta_0)}{\lambda+1}\right)^{\frac{1}{K}}
\leq \left( 1-\min \Big( 1,\frac{\mu_F(v)}{\mu} \Big)\frac{\lambda-\mu(1-\theta_0)}{\lambda+1} \right)^{\frac{\theta_0 \sqrt{\mu}}{2\sqrt{3}\sqrt{\lambda}}}
\end{align}

Then by \eqref{a:boundmax} and \eqref{afgert},
\begin{eqnarray*}
\left(\max\big( \sigma, 1 -\sigma m_K(\mu_F(v)) \big)^{1/K}\right)^{k-K} 
&\leq& \left( 1-\min\Big(1,\frac{\mu_F(v)}{\mu}\Big)\frac{\lambda-\mu(1-\theta_0)}{\lambda+1}\right)^{\frac{k\theta_0 \sqrt{\mu}}{2\sqrt{3}\sqrt{\lambda}}-1}
\end{eqnarray*}
which is the inequality we wanted to prove.
\end{proof}

\begin{proof}[proof of Corollary~\ref{corok}]
Taking $\lambda = 1+\mu$ in Proposition~\ref{prop:choose_sigma_and_K_appendix},
we can see that we have the result if
\[
\Exp[\Delta(x_k)] \leq   \left(1- \min \left(\frac{\mu_F(v)}{\mu} ,1\right)\frac{1+\mu\theta_0}{2+\mu} \right) ^{\frac{k\theta_0 \sqrt{\mu}}{2\sqrt{3}\sqrt{1+\mu}}-1}\Delta(x_0) \leq \epsilon.
\]
Passing to the logarithm leads to 
\begin{align*}
&\log\left(1- \min \left(\frac{\mu_F(v)}{\mu} ,1\right)\frac{1+\mu\theta_0}{2+\mu} \right) \left(\frac{k\theta_0 \sqrt{\mu}}{2\sqrt{3}\sqrt{1+\mu}}-1\right) + \log(\Delta(x_0)) \leq \log(\epsilon),
\end{align*}
which is equivalent to
\begin{align*}
&  \log\left(\frac{\Delta(x_0)}{\epsilon} \right) \leq -\log\left(1- \min \left(\frac{\mu_F(v)}{\mu} ,1\right)\frac{1+\mu\theta_0}{2+\mu} \right) \left(\frac{k\theta_0 \sqrt{\mu}}{2\sqrt{3}\sqrt{1+\mu}}-1\right)\enspace.
\end{align*}
As $-\log(1-x) \geq x$, it is enough to have
\[
\log\left(\frac{\Delta(x_0)}{\epsilon} \right) \leq \min \left(\frac{\mu_F(v)}{\mu} ,1\right)\frac{1+\mu\theta_0}{2+\mu} \left(\frac{k\theta_0 \sqrt{\mu}}{2\sqrt{3}\sqrt{1+\mu}}-1\right) \enspace,
\]
which yields:
\[
k\geq \frac{2\sqrt{3}}{\theta_0} \frac{(2+\mu)\sqrt{1+\mu}}{1+\mu \theta_0}  \max\left(\frac{\sqrt{\mu}}{\mu_F(v)},\frac{1}{\sqrt{\mu}}\right)\log\left(\frac{\Delta(x_0)}{\epsilon} \right)+\frac{2\sqrt{3}}{\theta_0}\sqrt{1+\frac{1}{\mu}}
\]
We get the corollary by noting that
$$
\frac{(2+\mu)\sqrt{1+\mu}}{1+\mu \theta_0} \leq 3\sqrt{2}.
$$
\end{proof}

\section*{Acknowledgement}

This work was supported by the
EPSRC  Grant EP/K02325X/1
{\em Accelerated Coordinate Descent Methods for Big Data Optimization},
the Centre for Numerical Algorithms and Intelligent Software (funded by EPSRC grant EP/G036136/1 and the Scottish Funding Council) and the Orange/Telecom ParisTech think tank Phi-TAB. 
This research was conducted using the HKU Information Technology Services research computing facilities that are supported in part by the Hong Kong UGC Special Equipment Grant (SEG HKU09).
We would like to thank Peter Richt\'arik for his
useful advice at the beginning of this project.

\bibliographystyle{../siamplain}
\bibliography{../literature}

\end{document}